\pdfoutput=1
\documentclass[11pt]{article}%
\usepackage{amsmath}
\usepackage{amsfonts}
\usepackage{amssymb}
\usepackage{graphicx}%
\newtheorem{theorem}{Theorem}

\newtheorem{corollary}[theorem]{Corollary}

\newtheorem{definition}[theorem]{Definition}

\newtheorem{lemma}[theorem]{Lemma}

\newtheorem{proposition}[theorem]{Proposition}
\newtheorem{remark}[theorem]{Remark}

\newenvironment{proof}[1][Proof]{\noindent\textbf{#1.} }{\ \rule{0.5em}{0.5em}}
\begin{document}

\title{A positive-definite inner product for vector-valued Macdonald polynomials}
\author{Charles F. Dunkl\thanks{E-mail address: \texttt{cfd5z@virginia.edu}}\\Dept. of Mathematics \\University of Virginia \\Charlottesville VA}
\date{15 August 2018}
\maketitle

\begin{abstract}
In a previous paper J.-G. Luque and the author \{Sem. Loth. Combin. 2011)
developed the theory of nonsymmetric Macdonald polynomials taking values in an
irreducible module of the Hecke algebra of the symmetric group $\mathcal{S}%
_{N}$. The polynomials are parametrized by $\left(  q,t\right)  $ and are
simultaneous eigenfunctions of a commuting set of Cherednik operators, which
were studied by Baker and Forrester (IMRN 1997). In the Dunkl-Luque paper
there is a construction of a pairing between $\left(  q^{-1},t^{-1}\right)  $
polynomials and $\left(  q,t\right)  $ polynomials, and for which the
Macdonald polynomials form a biorthogonal set. The present work is a sequel
with the purpose of constructing a symmetric bilinear form for which the
Macdonald polynomials form an orthogonal basis and to determine the region of
$\left(  q,t\right)  $-values for which the form is positive-definite.
Irreducible representations of the Hecke algebra are characterized by
partitions of $N$. The positivity region depends only on the maximum
hook-length of the Ferrers diagram of the partition.

\end{abstract}

\section{Introduction}

The theory of nonsymmetric Jack polynomials was generalized by Griffeth
\cite{G2010} to polynomials on the complex reflection groups of type $G\left(
n,p,N\right)  $ taking values in irreducible modules of the groups. This
theory simplifies somewhat for the group $G\left(  1,1,N\right)  $, the
symmetric group of $N$ objects, where any irreducible module is spanned by
standard Young tableaux all of the same shape, corresponding to a partition of
$N$. Luque and the author \cite{DL2012} developed an analogous theory for
vector-valued Macdonald polynomials taking values in irreducible modules of
the Hecke algebra of a symmetric group. The structure has parameters $\left(
q,t\right)  $ and depends on a commuting set of Cherednik operators whose
simultaneous eigenfunctions are the aforementioned Macdonald polynomials. The
paper showed how to construct the polynomials by means of a Yang-Baxter graph
(see Lascoux \cite{L2001}). Also a bilinear form was defined which paired
polynomials for the parameters $\left(  q^{-1},t^{-1}\right)  $ with those
parametrized by $\left(  q,t\right)  $ and resulted in biorthogonality
relations for the Macdonald polynomials. The present paper is a sequel whose
aim is to define a symmetric bilinear form for which these polynomials are
mutually orthogonal. Some other natural conditions are imposed on the form to
force uniqueness. The form is positive-definite for a $\left(  q,t\right)
$-region determined by the specific module.

For purposes of illustration the form is first defined for the scalar case,
and leads to expressions only slightly different from the well-known
hook-product formulas. Section \ref{VVMD} contains a short outline of
representation theory of the Hecke algebra, the Yang-Baxter graph of
vector-valued Macdonald polynomials and the process leading to the definition
of the symmetric bilinear form, followed by the characterization of $\left(
q,t\right)  $-values yielding positivity of the form. The details of the
construction of the polynomials and related operators along with the proofs of
their properties are found in \cite{DL2012}.

\subsection{Notation}

Let $\mathbb{N}_{0}:=\left\{  0,1,2,3,\ldots\right\}  $. The elements of
$\mathbb{N}_{0}^{N}$ are called \textit{compositions,} and for\textit{
}$\alpha=\left(  \alpha_{1},\ldots,\alpha_{N}\right)  \in\mathbb{N}_{0}^{N}$
let $\left\vert \alpha\right\vert :=\sum_{i=1}^{N}\alpha_{i}$. Let
$\mathbb{N}_{0}^{N,+}$ denote the set of partitions $\left\{  \lambda
\in\mathbb{N}_{0}^{N}:\lambda_{1}\geq\lambda_{2}\geq\ldots\geq\lambda
_{N}\right\}  $ and let $\alpha^{+}$ denote the nonincreasing rearrangement of
$\alpha$, for example if $\alpha=\left(  1,2,1,4\right)  $ then $\alpha
^{+}=\left(  4,2,1,1\right)  $. There are two partial orders on compositions
used in this work: for $\alpha,\beta\in\mathbb{N}_{0}^{N}$ the relation
$\alpha\succ\beta$ means $\alpha\neq\beta$ and $\sum_{i=1}^{j}\left(
\alpha_{i}-\beta_{i}\right)  \geq0$ for $1\leq j\leq N$ (the dominance order),
and $\alpha\vartriangleright\beta$ means $\left\vert \alpha\right\vert
=\left\vert \beta\right\vert $ and $\alpha^{+}\succ\beta^{+}$, or $\alpha
^{+}=\beta^{+}$ and $\alpha\succ\beta$. For a composition $\alpha\in
\mathbb{N}_{0}^{N}$ the inversion number is $\mathrm{inv}\left(
\alpha\right)  :=\#\left\{  \left(  i,j\right)  :1\leq i<j\leq N,\alpha
_{i}<\alpha_{j}\right\}  $. If $\alpha_{i}<\alpha_{i+1}$ then $\mathrm{inv}%
\left(  \alpha.s_{i}\right)  =\mathrm{inv}\left(  \alpha\right)  -1$. The rank
function for $\alpha\in\mathbb{N}_{0}^{N}$ is%
\[
r_{\alpha}\left(  i\right)  :=\#\left\{  j:\alpha_{j}>\alpha_{i}\right\}
+\#\left\{  j:1\leq j\leq i,\alpha_{j}=\alpha_{i}\right\}  ,~1\leq i\leq N;
\]
then $\alpha=\alpha^{+}$ if and only if $r_{\alpha}\left(  i\right)  =i$ for
all $i$.

The symmetric group $\mathcal{S}_{N}$ is generated by the adjacent
transpositions $s_{i}:=\left(  i,i+1\right)  $ for $1\leq i<N$, where $s_{i}$
acts on an $N$-tuple $a=\left(  a_{1},\ldots,a_{N}\right)  $ by $a.s_{i}%
=\left(  \ldots,a_{i+1},a_{i},\ldots\right)  $, interchanging entries $\#i$
and $\#\left(  i+1\right)  $.

The space of polynomials is $\mathcal{P}:=\mathbb{K}\left[  x_{1},x_{2}%
,\ldots,x_{N}\right]  $, where $\mathbb{K}:=\mathbb{Q}\left(  q,t\right)  $
and $q,t$ are transcendental or generic, that is, complex numbers satisfying
$q\neq1,q^{a}t^{b}\neq1$ for $a,b\in\mathbb{Z}$ and $-N\leq b\leq N$. For
$\alpha\in\mathbb{N}_{0}^{N}$ the \textit{monomial} $x^{\alpha}:=\prod
\limits_{i=1}^{N}x_{i}^{\alpha_{i}}$. The polynomials homogeneous of degree
$n$ are defined as $\mathcal{P}_{n}=\mathrm{span}_{\mathbb{K}}\left\{
x^{\alpha}:\left\vert \alpha\right\vert =n\right\}  $ for $n=0,1,2,\ldots$.
The group $\mathcal{S}_{N}$ acts on polynomials by permutation of coordinates,
$p\left(  x\right)  \rightarrow\left(  ps_{i}\right)  \left(  x\right)
:=p\left(  x.s_{i}\right)  $.

The Hecke algebra $\mathcal{H}_{N}\left(  t\right)  $ is the associative
algebra generated by\linebreak\ $\left\{  T_{1},T_{2},\ldots,T_{N-1}\right\}
$ subject to the relations%
\begin{gather}
\left(  T_{i}+1\right)  \left(  T_{i}-t\right)  =0,\label{Hrel}\\
T_{i}T_{i+1}T_{i}=T_{i+1}T_{i}T_{i+1},1\leq i\leq N-2,\nonumber\\
T_{i}T_{j}=T_{j}T_{i},1\leq i<j-1\leq N-2.\nonumber
\end{gather}
The quadratic relation implies $T_{i}^{=1}=\frac{1}{t}\left(  T_{i}%
+1-t\right)  \in\mathcal{H}_{N}\left(  t\right)  $. For generic $t$ there is a
linear isomorphism $\mathbb{K}\mathcal{S}_{N}\rightarrow\mathcal{H}_{N}\left(
t\right)  $ generated by $s_{i}\rightarrow T_{i}$.

For $p\in\mathcal{P}$ and $1\leq i<N$ define%
\begin{equation}
p\left(  x\right)  T_{i}:=\left(  1-t\right)  x_{i+1}\frac{p\left(  x\right)
-p\left(  x.s_{i}\right)  }{x_{i}-x_{i+1}}+tp\left(  x.s_{i}\right)  .
\label{pxTi}%
\end{equation}
It can be shown straightforwardly that these operators satisfy the defining
relations of $\mathcal{H}_{N}\left(  t\right)  $. Also $ps_{i}=p$ (symmetry in
$\left(  x_{i},x_{i+1}\right)  $) if and only if $pT_{i}=tp$ (because
$pT_{i}-tp=\dfrac{tx_{i}-x_{i+1}}{x_{i}-x_{i+1}}\left(  p-ps_{i}\right)  $),
and $pT_{i}=-p$ if and only if $p\left(  x\right)  =\left(  tx_{i}%
-x_{i+1}\right)  p_{0}\left(  x\right)  $ where $p_{0}\in\mathcal{P}$ and
$p_{0}s_{i}=p_{0}$. Also $x_{i}T_{i}=x_{i+1}$ and $1T_{i}=t$.

\section{\label{SMD}Scalar nonsymmetric Macdonald polynomials}

For $f\in\mathcal{P}$ define shift, Cherednik and Dunkl operators by (see
Baker and Forrester \cite{BF1997}, also \cite{DL2012})%
\begin{align}
fw\left(  x\right)   &  :=f\left(  qx_{N},x_{1},x_{2},\ldots,x_{N-1}\right)
,\nonumber\\
f\xi_{i}  &  :=t^{i-1}fT_{i-1}^{-1}T_{i-2}^{-1}\cdots T_{1}^{-1}%
wT_{N-1}T_{N-2}\cdots T_{i}\label{defxi0}\\
f\mathcal{D}_{N}  &  :=\left(  f-f\xi_{N}\right)  /x_{N},~f\mathcal{D}%
_{i}:=\frac{1}{t}fT_{i}\mathcal{D}_{i+1}T_{i}.\nonumber
\end{align}
Note $\xi_{i}=\frac{1}{t}T_{i}\xi_{i+1}T_{i}$. It is a nontrivial result that
$D_{i}$ maps $\mathcal{P}_{n}$ to $\mathcal{P}_{n-1}$. The operators $\xi_{i}$
commute with each other and there is a basis of simultaneous eigenfunctions,
the nonsymmetric Macdonald polynomials $M_{\alpha}$, labeled by $\alpha
\in\mathbb{N}_{0}^{N}$ with $\vartriangleright$-leading term $\,q^{\ast
}t^{\ast}x^{\alpha}$ ($q^{\ast}t^{\ast}$ denotes integer powers of $q,t$) such
that%
\begin{align*}
M_{\alpha}\left(  x\right)   &  =q^{\ast}t^{\ast}x^{\alpha}+\sum
_{\alpha\vartriangleright\beta}A_{\alpha\beta}\left(  q,t\right)  x^{\beta}\\
M_{\alpha}\xi_{i}  &  =q^{\alpha_{i}}t^{N-r_{\alpha}\left(  i\right)
}M_{\alpha},1\leq i\leq N;
\end{align*}
where the coefficients $A_{\alpha\beta}\left(  q,t\right)  $ are rational
functions of $\left(  q,t\right)  $ whose denominators are of the form
$\left(  1-q^{a}t^{b}\right)  $. The \textit{spectral vector} is
$\zeta_{\alpha}\left(  i\right)  =q^{\alpha_{i}}t^{N-r_{\alpha}\left(
i\right)  }$, $1\leq i\leq N$. There is a simple relation between $M_{\alpha}$
and $M_{\alpha.s_{i}}$ when $\alpha_{i}<\alpha_{i+1}$ and $\rho=\zeta_{\alpha
}\left(  i+1\right)  /\zeta_{\alpha}\left(  i\right)  =q^{\alpha_{i+1}%
-\alpha_{i}}t^{r_{\alpha}\left(  i\right)  -r_{\alpha}\left(  i+1\right)  }$%
\begin{align}
M_{\alpha}T_{i}  &  =M_{\alpha s_{i}}-\frac{1-t}{1-\rho}M_{\alpha}%
,\label{MT1}\\
M_{\alpha.s_{i}}T_{i}  &  =\frac{\left(  1-\rho t\right)  \left(
t-\rho\right)  }{\left(  1-\rho\right)  ^{2}}M_{\alpha}+\frac{\rho\left(
1-t\right)  }{\left(  1-\rho\right)  }M_{\alpha.s_{i}}, \label{MT2}%
\end{align}
and $\zeta_{\alpha.s_{i}}=\zeta_{\alpha}.s_{i}$. If $\alpha_{i}=\alpha_{i+1}$
then%
\begin{equation}
M_{\alpha}T_{i}=tM_{\alpha}. \label{MT3}%
\end{equation}

The other step needed to construct any $M_{\alpha}$ starting from $1$ is the
affine step%
\begin{align}
M_{\alpha\Phi}  &  =x_{N}\left(  M_{\alpha}w\right)  ,\label{MT4}\\
\alpha\Phi &  :=\left(  \alpha_{2},\alpha_{3},\ldots,\alpha_{N},\alpha
_{1}+1\right)  ,\nonumber\\
\zeta_{\alpha\Phi}  &  =\left(  \zeta_{\alpha}\left(  2\right)  ,\ldots
,\zeta_{\alpha}\left(  N\right)  ,q\zeta_{\alpha}\left(  1\right)  \right)
\nonumber
\end{align}

Formulas (\ref{MT1}) and (\ref{MT4}) can be interpreted as edges of a
Yang-Baxter graph for generating the polynomials (see Lascoux \cite[Sect.
9]{L2001}). This graph has the root $\left(  \boldsymbol{0},\left[
t^{N-i}\right]  _{i=1}^{N},1\right)  $ and nodes $\left(  \alpha,\zeta
_{\alpha},M_{\alpha}\right)  $. There are steps $\left(  \alpha,\zeta_{\alpha
},M_{\alpha}\right)  \overset{s_{i}}{\rightarrow}\left(  \alpha.s_{i}%
,\zeta_{\alpha}.s_{i},M_{\alpha.s_{i}}\right)  $ for $\alpha_{i+1}>\alpha_{i}$
given by $M_{\alpha s_{i}}=M_{\alpha}T_{i}+\frac{1-t}{1-\zeta_{\alpha}\left(
i+1\right)  /\zeta_{\alpha}\left(  i\right)  }M_{\alpha}$ and affine steps
$\left(  \alpha,\zeta_{\alpha},M_{\alpha}\right)  \overset{\Phi}{\rightarrow
}\left(  \alpha\Phi,\zeta_{\alpha\Phi},M_{\alpha\Phi}\right)  $ (given by
(\ref{MT4}) ).

There is a short proof using Macdonald polynomials that $\mathcal{D}_{N}$ maps
$\mathcal{P}_{N}$ to $\mathcal{P}_{N-1}$: when $\alpha_{N}=0$ then $r_{\alpha
}\left(  N\right)  =N,\zeta_{\alpha}\left(  N\right)  =1$ and $M_{\alpha}%
\xi_{N}=M_{\alpha}$ thus $M_{\alpha}\mathcal{D}_{N}=0$; if $\alpha_{N}\geq1$
then the raising (affine) formula is $M_{\alpha}\left(  x\right)
=x_{N}M_{\beta}w\left(  x\right)  $ where $\beta=\left(  \alpha_{N}%
-1,\alpha_{1},\alpha_{2},\ldots,\alpha_{N-1}\right)  $ thus\linebreak%
\ $M_{\alpha}\left(  1-\xi_{N}\right)  =\left(  1-\zeta_{\alpha}\left(
N\right)  \right)  M_{\alpha}$ which is divisible by $x_{N}$.

Our logical outline is to first state a number of hypotheses to be satisfied
by the inner product, then deduce consequences leading to a formula which is
used as a definition. To finish one has to show that the hypotheses are
satisfied. The presentation is fairly sketchy for the scalar case which is
mostly intended as illustration. The material for vector-valued Macdonald
polynomials is more detailed.

The \textbf{hypotheses} (BF1) for the symmetric bilinear form $\left\langle
\cdot,\cdot\right\rangle $ on $\mathcal{P}$, with $w^{\ast}:=T_{N-1}%
^{-1}\cdots T_{1}^{-1}wT_{N-1}\cdots T_{1}$, are (for $f,g\in\mathcal{P}$)
\begin{align}
\left\langle 1,1\right\rangle  &  =1,\label{BF1.1}\\
\left\langle fT_{i},g\right\rangle  &  =\left\langle f,gT_{i}\right\rangle
,~1\leq i<N\label{BF1.2}\\
\left\langle f\xi_{N},g\right\rangle  &  =\left\langle f,g\xi_{N}\right\rangle
,\label{BF1.3}\\
\left\langle f\mathcal{D}_{N},g\right\rangle  &  =\left(  1-q\right)
\left\langle f,x_{N}\left(  gw^{\ast}w\right)  \right\rangle . \label{BF1.4}%
\end{align}
From the definition of $w=T_{N-1}^{-1}\cdots T_{1}^{-1}\xi_{1}$ it follows
that%
\begin{align*}
\left\langle f,gw^{\ast}\right\rangle  &  =\left\langle f,gT_{N-1}^{-1}\cdots
T_{1}^{-1}\xi_{1}\right\rangle =\left\langle f\xi_{1},gT_{N-1}^{-1}\cdots
T_{1}^{-1}\right\rangle \\
&  =\left\langle f\xi_{1}T_{1}\cdots T_{N-1},g\right\rangle =\left\langle
fw,g\right\rangle .
\end{align*}
(It is a trivial exercise to show $\left\langle fT_{i}^{-1},g\right\rangle
=\left\langle f,gT_{i}^{-1}\right\rangle $.) Technically there is a problem in
defining the adjoint of a general operator on $\mathcal{P}$ because
$\mathcal{P}$ is infinite dimensional and it need not be true that the map
$f\mapsto\left\langle f,\cdot\right\rangle $ is one-to-one into the dual
space; here $w^{\ast}$ is taken as a symbolic name. It follows from
(\ref{BF1.2}), (\ref{BF1.3}) and $\xi_{i}=\frac{1}{t}T_{i}\xi_{i+1}T_{i}$ that
$\left\langle f\xi_{i},g\right\rangle =\left\langle f,g\xi_{i}\right\rangle $
for all $f,g\in\mathcal{P}$ and all $i$. This implies the mutual orthogonality
of $\left\{  M_{\alpha}:\alpha\in\mathbb{N}_{0}^{N}\right\}  $ because the
spectral vector $\zeta_{\alpha}$ determines $\alpha$. Implicitly
$t\in\mathbb{R}$ since the eigenvalues of~$T_{i}$ are $t,-1$. If $\deg
f\neq\deg g$ then $\left\langle f,g\right\rangle =0$ because the Macdonald
polynomials form a homogeneous basis. For convenience denote $\left\langle
f,f\right\rangle =\left\Vert f\right\Vert ^{2}$ (no claim is being made about positivity).

\begin{definition}
For $z\in\mathbb{K}$ let%
\[
u\left(  z\right)  :=\frac{\left(  t-z\right)  \left(  1-zt\right)  }{\left(
1-z\right)  ^{2}}.
\]

\end{definition}

Note that $u\left(  z^{-1}\right)  =u\left(  z\right)  $.

\begin{proposition}
Suppose (BF1) holds, $a_{i}<\alpha_{i+1}$ and $\rho=q^{\alpha_{i+1}-\alpha
_{i}}t^{r_{\alpha}\left(  i\right)  -r_{\alpha}\left(  i+1\right)  }$ then%
\[
\left\Vert M_{\alpha.s_{i}}\right\Vert ^{2}=u\left(  \rho\right)  \left\Vert
M_{\alpha}\right\Vert ^{2}%
\]

\end{proposition}

\begin{proof}
From equation (\ref{MT1}) $\left\langle M_{\alpha}T_{i},M_{\alpha.s_{i}%
}\right\rangle =\left\Vert M_{\alpha.s_{i}}\right\Vert ^{2}$ (by hypothesis
$\left\langle M_{\alpha},M_{\alpha.s_{i}}\right\rangle =0$) and by equation
(\ref{MT2}) $\left\langle M_{\alpha}T_{i},M_{\alpha.s_{i}}\right\rangle
=\left\langle M_{\alpha},M_{\alpha.s_{i}}T_{i}\right\rangle =\frac{\left(
1-\rho t\right)  \left(  t-\rho\right)  }{\left(  1-\rho\right)  ^{2}%
}\left\Vert M_{\alpha}\right\Vert ^{2}=u\left(  \rho\right)  \left\Vert
M_{\alpha}\right\Vert ^{2}$.
\end{proof}

\begin{definition}
For $\alpha\in\mathbb{N}_{0}^{N}$ let%
\[
\mathcal{E}\left(  \alpha\right)  :=\prod\limits_{1\leq i<j\leq N,\alpha
_{i}<\alpha_{j}}u\left(  q^{\alpha_{j}-\alpha_{i}}t^{r_{\alpha}\left(
i\right)  -r_{\alpha}\left(  j\right)  }\right)  .
\]

\end{definition}

\begin{proposition}
\label{e-norm1}Suppose (BF1) holds and $\alpha\in\mathbb{N}_{0}^{N}$ then
$\left\Vert M_{\alpha^{+}}\right\Vert ^{2}=\mathcal{E}\left(  \alpha\right)
\left\Vert M_{\alpha}\right\Vert ^{2}$.
\end{proposition}

\begin{proof}
Arguing by induction on $\mathrm{inv}\left(  \alpha\right)  $ it suffices to
show that $\alpha_{i}<\alpha_{i+1}$ implies $\mathcal{E}\left(  \alpha\right)
/\mathcal{E}\left(  \alpha.s_{i}\right)  =u\left(  q^{\alpha_{i+1}-\alpha_{i}%
}t^{r_{\alpha}\left(  i\right)  -r_{\alpha}\left(  i+1\right)  }\right)  $.
The factors corresponding to pairs $\left(  l,j\right)  $ with $l,j\neq i,i+1$
are the same in the products, and the pairs with just one of $i,i+1$ are
interchanged in $\mathcal{E}\left(  \alpha\right)  ,\mathcal{E}\left(
\alpha.s_{i}\right)  $. There is only one factor in $\mathcal{E}\left(
\alpha\right)  $ that is not in $\mathcal{E}\left(  \alpha.s_{i}\right)  $,
namely $u\left(  q^{\alpha_{i+1}-\alpha_{i}}t^{r_{\alpha}\left(  i\right)
-r_{\alpha}\left(  i+1\right)  }\right)  $ coming from $\left(  i,i+1\right)
$. Thus $\mathcal{E}\left(  \alpha\right)  \left\Vert M_{\alpha}\right\Vert
^{2}=\mathcal{E}\left(  \alpha.s_{i}\right)  \left\Vert M_{\alpha.s_{i}%
}\right\Vert ^{2}$.
\end{proof}

\begin{lemma}
Suppose $\alpha\in\mathbb{N}_{0}^{N}$ then $M_{\alpha\Phi}\mathcal{D}%
_{N}=\left(  1-q\zeta_{\alpha}\left(  1\right)  \right)  M_{\alpha}w$.
\end{lemma}

\begin{proof}
By definition $M_{\alpha\Phi}\mathcal{D}_{N}=\left(  1/x_{N}\right)
M_{\alpha\Phi}\left(  1-\xi_{N}\right)  =\left(  1/x_{N}\right)  \left(
1-\zeta_{\alpha\Phi}\left(  N\right)  \right)  M_{\alpha\Phi}=\left(
1-q\zeta_{\alpha}\left(  1\right)  \right)  M_{\alpha}w$.
\end{proof}

\begin{remark}
It is incompatible with (\ref{BF1.1}), (\ref{BF1.2}) and (\ref{BF1.3}) to
require either $\left\langle f\mathcal{D}_{N},g\right\rangle =c\left\langle
f,x_{N}g\right\rangle $ with some constant $c$, or $\left\langle x_{N}%
f,x_{N}g\right\rangle =\left\langle f,g\right\rangle $ . Let $f=M_{\alpha\Phi
}$ and $g=M_{\beta}w$ with $\left\vert \alpha\right\vert =\left\vert
\beta\right\vert +1$; then $\left\langle f,x_{N}g\right\rangle =\left\langle
M_{\alpha\Phi},M_{\beta\Phi}\right\rangle $ while $\left\langle f\mathcal{D}%
_{N},g\right\rangle =\left(  1-q\zeta_{\alpha}\left(  1\right)  \right)
\left\langle M_{\alpha}w,M_{\beta}w\right\rangle =\left(  1-q\zeta_{\alpha
}\left(  1\right)  \right)  \left\langle M_{\alpha},M_{\beta}ww^{\ast
}\right\rangle $. If $\alpha\neq\beta$ then $\left\langle f,x_{N}%
g\right\rangle =0$ but in general $\left\langle M_{\alpha},M_{\beta}ww^{\ast
}\right\rangle \neq0$; for example $\alpha=\left(  1,0,0,0\right)  $ and
$\beta=\left(  0,1,0,0\right)  $. For the second part let $f=M_{\alpha}w$ so
that $\left\langle x_{N}f,x_{N}g\right\rangle =\left\langle M_{\alpha\Phi
},M_{\beta\Phi}\right\rangle $, while $\left\langle f,g\right\rangle
=\left\langle M_{\alpha},M_{\beta}ww^{\ast}\right\rangle $.
\end{remark}

\begin{proposition}
\label{phinorm1}Suppose (BF1) holds and $\alpha\in\mathbb{N}_{0}^{N}$ then
$\left\Vert M_{\alpha\Phi}\right\Vert ^{2}=\dfrac{1-q\zeta_{\alpha}\left(
1\right)  }{1-q}\left\Vert M_{\alpha}\right\Vert ^{2}$.
\end{proposition}

\begin{proof}
Let $g\in\mathcal{P}$ with $\deg g=\left\vert \alpha\right\vert $ then by the
previous lemma
\[
\left\langle M_{\alpha\Phi}\mathcal{D}_{N},g\right\rangle =\left(
1-q\zeta_{\alpha}\left(  1\right)  \right)  \left\langle M_{\alpha
}w,g\right\rangle =\left(  1-q\zeta_{\alpha}\left(  1\right)  \right)
\left\langle M_{\alpha},gw^{\ast}\right\rangle .
\]
Specialize to $gw^{\ast}=M_{\alpha}$ to obtain%
\[
\left\langle M_{\alpha\Phi}\mathcal{D}_{N},M_{\alpha}\left(  w^{\ast}\right)
^{-1}\right\rangle =\left(  1-q\zeta_{\alpha}\left(  1\right)  \right)
\left\langle M_{\alpha},M_{\alpha}\right\rangle .
\]
By (\ref{BF1.4}) $\left\langle M_{\alpha\Phi}\mathcal{D}_{N},M_{\alpha}\left(
w^{\ast}\right)  ^{-1}\right\rangle =\left(  1-q\right)  \left\langle
M_{\alpha\Phi},x_{N}\left(  M_{\alpha}\left(  w^{\ast}\right)  ^{-1}\right)
w^{\ast}w\right\rangle =\left(  1-q\right)  \left\langle M_{\alpha\Phi
},M_{\alpha\Phi}\right\rangle $. This completes the proof.
\end{proof}

Next we use (BF1) to derive a formula for $\left\Vert M_{\lambda}\right\Vert
^{2}$ for any $\lambda\in\mathbb{N}_{0}^{N,+}$. Suppose $\lambda_{m}\geq1$ and
$\lambda_{i}=0$ for $i>m$. Let $\alpha=\left(  \lambda_{1},\ldots
,\lambda_{m-1},0,\ldots0,\lambda_{m}\right)  $, $\beta=\left(  \lambda
_{m}-1,\lambda_{1},\ldots,\lambda_{m-1},0,\ldots\right)  $ so that
$\alpha=\beta\Phi$, and $\gamma=\beta^{+}=\left(  \lambda_{1},\ldots
,\lambda_{m-1},\lambda_{m}-1,0,\ldots\right)  $. Then $\left\Vert M_{\lambda
}\right\Vert ^{2}=\mathcal{E}\left(  \alpha\right)  \left\Vert M_{\alpha
}\right\Vert ^{2}$, $\left\Vert M_{\alpha}\right\Vert ^{2}=\frac
{1-q\zeta_{\beta}\left(  1\right)  }{1-q}\left\Vert M_{\beta}\right\Vert ^{2}$
and $\left\Vert M_{\gamma}\right\Vert ^{2}=\mathcal{E}\left(  \beta\right)
\left\Vert M_{\beta}\right\Vert ^{2}$. The rank vectors for $\alpha,\beta$ are
$\left(  \ldots,m+1,\ldots,N,m\right)  $ and $\left(  m,1,2,\ldots
,m-1,m+1\ldots\right)  $ respectively. Then%
\begin{align*}
\mathcal{E}\left(  \alpha\right)   &  =\prod\limits_{i=m+1}^{N}u\left(
q^{\lambda_{m}}t^{i-m}\right)  =t^{N-m}\frac{\left(  1-q^{\lambda_{m}}\right)
\left(  1-q^{\lambda_{m}}t^{N-m+1}\right)  }{\left(  1-q^{\lambda_{m}%
}t\right)  \left(  1-q^{\lambda_{m}}t^{N-m}\right)  },\\
\zeta_{\beta}\left(  1\right)   &  =q^{\lambda_{m}-1}t^{N-m},\\
\mathcal{E}\left(  \beta\right)   &  =\prod\limits_{i=1}^{m-1}u\left(
q^{\lambda_{i}-\lambda_{m}+1}t^{m-i}\right)
\end{align*}
and%
\[
\left\Vert M_{\lambda}\right\Vert ^{2}=\frac{1-q^{\lambda_{m}}t^{N-m}}%
{1-q}\frac{\mathcal{E}\left(  \alpha\right)  }{\mathcal{E}\left(
\beta\right)  }\left\Vert M_{\gamma}\right\Vert ^{2}.
\]
(The product $\mathcal{E}\left(  \alpha\right)  $ telescopes. If $m=N$ then
$\mathcal{E}\left(  \alpha\right)  =1.$) This is the key ingredient for an
inductive argument. Denote the transpose of (the Ferrers diagram) $\lambda
\in\mathbb{N}_{0}^{N,+}$ by $\lambda^{\prime}$, so that $\mathrm{arm}\left(
\lambda;i,j\right)  =\lambda_{i}-j$ and $\mathrm{leg}\left(  \lambda
;i,j\right)  =\lambda_{j}^{\prime}-i$ and the hook product%
\[
h_{q,t}\left(  \lambda;z\right)  :=\prod\limits_{\left(  i,j\right)
\in\lambda}\left(  1-zq^{\mathrm{arm}\left(  i,j\right)  }t^{\mathrm{leg}%
\left(  i,j\right)  }\right)  .
\]
The changes in the hook product going from $\lambda$ to $\gamma$ come from the
hooks at $\left\{  \left(  i,\lambda_{m}\right)  :1\leq i\leq m-1\right\}  $
and $\left\{  \left(  m,j\right)  :1\leq j\leq\lambda_{m}\right\}  $, thus%
\[
\frac{h_{q,t}\left(  \lambda;z\right)  }{h_{q,t}\left(  \gamma;z\right)
}=\prod\limits_{i=1}^{m-1}\frac{1-zq^{\lambda_{i}-\lambda_{m}}t^{m-i}%
}{1-zq^{\lambda_{i}-\lambda_{m}}t^{m-i-1}}\left(  1-zq^{\lambda_{m}-1}\right)
,
\]
because $\prod\limits_{j=1}^{\lambda_{m}-1}\frac{1-zq^{\lambda_{m}-j}%
}{1-zq^{\lambda_{m}-j-1}}\left(  1-z\right)  =1-zq^{\lambda_{m}-1}$ by
telescoping (this telescoping property is unique to the scalar case and the
norm formulas for the vector-valued case look quite different). Furthermore%
\[
\frac{h_{q,t}\left(  \lambda;tz\right)  }{h_{q,t}\left(  \gamma;tz\right)
}\frac{h_{q,t}\left(  \gamma;z\right)  }{h_{q,t}\left(  \lambda;z\right)
}=t^{1-m}\prod\limits_{i=1}^{m-1}u\left(  zq^{\lambda_{i}-\lambda_{m}}%
t^{m-i}\right)  \frac{1-ztq^{\lambda_{m}-1}}{1-zq^{\lambda_{m}-1}}.
\]
Set $z=q$ to obtain%
\[
\mathcal{E}\left(  \beta\right)  =t^{m-1}\left(  \frac{1-q^{\lambda_{m}}%
}{1-tq^{\lambda_{m}}}\right)  \frac{h_{q,t}\left(  \lambda;tq\right)
}{h_{q,t}\left(  \gamma;tq\right)  }\frac{h_{q,t}\left(  \gamma;q\right)
}{h_{q,t}\left(  \lambda;q\right)  }%
\]
and%
\begin{align*}
\frac{\left\Vert M_{\lambda}\right\Vert ^{2}}{\left\Vert M_{\gamma}\right\Vert
^{2}}  &  =t^{1-m}\left(  \frac{1-q^{\lambda_{m}}t^{N-m}}{1-q}\right)  \left(
\frac{1-q^{\lambda_{m}}t}{1-q^{\lambda_{m}}}\right) \\
&  \times\frac{h_{q,t}\left(  \lambda;q\right)  }{h_{q,t}\left(
\gamma;q\right)  }\frac{h_{q,t}\left(  \gamma;tq\right)  }{h_{q,t}\left(
\lambda;tq\right)  }t^{N-m}\frac{\left(  1-q^{\lambda_{m}}\right)  \left(
1-q^{\lambda_{m}}t^{N-m+1}\right)  }{\left(  1-q^{\lambda_{m}}t\right)
\left(  1-q^{\lambda_{m}}t^{N-m}\right)  }\\
&  =t^{N-2m+1}\left(  \frac{1-q^{\lambda_{m}}t^{N-m+1}}{1-q}\right)
\frac{h_{q,t}\left(  \lambda;q\right)  }{h_{q,t}\left(  \gamma;q\right)
}\frac{h_{q,t}\left(  \gamma;tq\right)  }{h_{q,t}\left(  \lambda;tq\right)  }.
\end{align*}

Define the generalized $q,t$ factorial for $\lambda\in\mathbb{N}_{0}^{N,+}$ by
$\left(  z;q,t\right)  _{\lambda}=\prod\limits_{i=1}^{N}\left(  zt^{1-i}%
;q\right)  _{\lambda_{i}}$ , where $\left(  z;q\right)  _{n}:=\prod
\limits_{i=1}^{N}\left(  1-zq^{i-1}\right)  $.

\begin{theorem}
Suppose (BF1) holds and $\lambda\in\mathbb{N}_{0}^{N,+}$ then%
\begin{align}
\left\Vert M_{\lambda}\right\Vert ^{2}  &  =t^{k\left(  \lambda\right)  }%
\frac{h_{q,t}\left(  \lambda;q\right)  }{h_{q,t}\left(  \lambda;qt\right)
}\frac{\left(  qt^{N-1};q\right)  _{\lambda}}{\left(  1-q\right)  ^{\left\vert
\lambda\right\vert }},\label{normdef1}\\
k\left(  \lambda\right)   &  :=\sum_{i=1}^{N}\left(  N-2i+1\right)
\lambda_{i}.\nonumber
\end{align}

\end{theorem}

\begin{proof}
The formula gives the trivial result $\left\Vert 1\right\Vert ^{2}=1$, where
$M_{\mathbf{0}}=1.$One need only check $\left(  qt^{N-1};q,t\right)
_{\lambda}/\left(  qt^{N-1};q,t\right)  _{\gamma}=\left(  qt^{N-m};q\right)
_{\lambda_{m}}/\left(  qt^{N-m};q\right)  _{\lambda_{m}-1}=1-q^{\lambda_{m}%
}t^{N-m}$ and $k\left(  \lambda\right)  -k\left(  \gamma\right)  =N-2m+1$.
\end{proof}

Note that $k\left(  \lambda\right)  =\sum_{i=1}^{\left\lfloor N/2\right\rfloor
}\left(  \lambda_{i}-\lambda_{N+1-i}\right)  \left(  N-2i+1\right)  \geq0$. We
now use formula (\ref{normdef1}), together with $\left\langle M_{\alpha
},M_{\beta}\right\rangle =0$ for $\alpha\neq\beta$ and $\left\Vert M_{\alpha
}\right\Vert ^{2}=\mathcal{E}\left(  \alpha\right)  ^{-1}\left\Vert
M_{\alpha^{+}}\right\Vert ^{2}$ as definition of the form. It is
straightforward to check properties (\ref{BF1.1}), (\ref{BF1.2}) and
(\ref{BF1.3}). For (\ref{BF1.4}) we need to show $\left\Vert M_{\alpha\Phi
}\right\Vert ^{2}=\frac{1-q\zeta_{\alpha}\left(  1\right)  }{1-q}\left\Vert
M_{\alpha}\right\Vert ^{2}$ (detailed argument in Section \ref{VVMD}) and the
formula $\left\langle f\mathcal{D}_{N},g\right\rangle =\left(  1-q\right)
\left\langle f,x_{N}\left(  gw^{\ast}w\right)  \right\rangle $. It suffices to
prove this for $f=M_{\gamma}$ and $gw^{\ast}=M_{\beta}$ with $\left\vert
\gamma\right\vert =\left\vert \beta\right\vert +1$; indeed $\left\langle
M_{\gamma}\mathcal{D}_{N},M_{\beta}\left(  w^{\ast}\right)  ^{-1}\right\rangle
=\left\langle M_{\gamma}\mathcal{D}_{N}w^{-1},M_{\beta}\right\rangle $ and
$\left\langle M_{\gamma},x_{N}M_{\beta}w\right\rangle =\left\langle M_{\gamma
},M_{\beta\Phi}\right\rangle $. If $\gamma=\alpha\Phi$ for some $\alpha$ then
both terms vanish for $\alpha\neq\beta$, otherwise the equation$\left\Vert
M_{\alpha\Phi}\right\Vert ^{2}=\frac{1-q\zeta_{\alpha}\left(  1\right)  }%
{1-q}\left\Vert M_{\alpha}\right\Vert ^{2}$ holds. If $\gamma_{N}=0$ then
$M_{\gamma}\mathcal{D}_{N}=0$ and $\left\langle M_{\gamma},M_{\beta\Phi
}\right\rangle =0$ (since $\gamma\neq\beta\Phi$).

The last of our concerns here is to determine the $\left(  q,t\right)  $
region of positivity of $\left\langle \cdot,\cdot\right\rangle $. Inspection
of the norm formula shows that there is an even number of factors of the form
$1-q^{a}t^{b}$ where $a\geq1$ and $0\leq b\leq N$. There are two
possibilities: either each such factor is positive or each is negative. Always
assume $q,t>0$ and $q\neq1$. If each is positive then $0<q<1$ and $q^{a}%
t^{b}\leq qt^{b}$. If $0<t\leq1$ then $qt^{b}\leq q<1$, or if $t>1$ then
$qt^{b}\leq qt^{N}<1$, that is $q<t^{-N}$. If each factor is negative then
$q>1$: if $t\geq1$ then $q^{a}t^{b}\geq q>1$, or if $0<t\leq1$ then
$q^{a}t^{b}\geq qt^{b}\geq qt^{N}>1$, that is, $q>t^{-N}$.

\begin{proposition}
The inner product $\left\langle \cdot,\cdot\right\rangle $ is
positive-definite, that is, $\left\langle M_{\alpha},M_{\alpha}\right\rangle
>0$ for all $\alpha\in\mathbb{N}_{0}^{N}$ provided $q,t>0$, and $q<\min\left(
1,t^{-N}\right)  $ or $q>\max\left(  1,t^{-N}\right)  $.
\end{proposition}

Figure \ref{posit1} is an illustration of the positivity region with $N=4$
using logarithmic coordinates.%

\begin{figure}
[ptb]
\begin{center}
\includegraphics[
height=3.1514in,
width=3.1514in
]%
{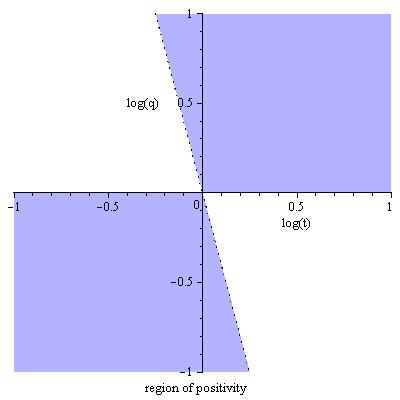}%
\caption{Logarithmic coordinates, N = 4}%
\label{posit1}%
\end{center}
\end{figure}

\section{\label{VVMD}Vector-valued Macdonald polynomials.}

These are polynomials whose values lie in an irreducible $\mathcal{H}%
_{N}\left(  t\right)  $-module. The generating relations for the Hecke algebra
$\mathcal{H}_{N}\left(  t\right)  $ are stated in (\ref{Hrel}). For the
purpose of constructing a positive symmetric bilinear form we restrict $t>0$.
Also throughout $q,t\neq0,1$.

\subsection{Representations of the Hecke algebra}

The irreducible modules of $\mathcal{H}_{N}\left(  t\right)  $ correspond to
partitions of $N$ and are constructed in terms of Young tableaux (see Dipper
and James \cite{DJ1986}).

Let $\tau$ be a partition of $N,$ that is, $\tau\in\mathbb{N}_{0}^{N,+}$ and
$\left\vert \tau\right\vert =N$. Thus $\tau=\left(  \tau_{1},\tau_{2}%
,\ldots\right)  $ and often the trailing zero entries are dropped when writing
$\tau$. The length of $\tau$ is $\ell\left(  \tau\right)  =\max\left\{
i:\tau_{i}>0\right\}  $. There is a Ferrers diagram of shape $\tau$ (given the
same label), with boxes at points $\left(  i,j\right)  $ with $1\leq i\leq
\ell\left(  \tau\right)  $ and $1\leq j\leq\tau_{i}$. A \textit{tableau} of
shape $\tau$ is a filling of the boxes with numbers, and a \textit{reverse
standard Young tableau} (RSYT) is a filling with the numbers $\left\{
1,2,\ldots,N\right\}  $ so that the entries decrease in each row and each
column. Denote the set of RSYT's of shape $\tau$ by $\mathcal{Y}\left(
\tau\right)  $ and let $V_{\tau}=\mathrm{span}_{\mathbb{K}}\left\{
S:S\in\mathcal{Y}\left(  \tau\right)  \right\}  $ with orthogonal basis
$\mathcal{Y}\left(  \tau\right)  $ (recall $\mathbb{K=Q}\left(  q,t\right)
$). Set $n_{\tau}:=\dim V_{\tau}=\#\mathcal{Y}\left(  \tau\right)  $ The
formula for the dimension is a hook-length product. For $1\leq i\leq N$ and
$S\in\mathcal{Y}\left(  \tau\right)  $ the entry $i$ is at coordinates
$\left(  \mathrm{\operatorname{row}}\left(  i,S\right)  ,\operatorname{col}%
\left(  i,S\right)  \right)  $ and the \textit{content} of the entry is
$c\left(  i,S\right)  :=\operatorname{col}\left(  i,S\right)
-\mathrm{\operatorname{row}}\left(  i,S\right)  $. Each $S\in\mathcal{Y}%
\left(  \tau\right)  $ is uniquely determined by its \textit{content vector}
$\left[  c\left(  i,S\right)  \right]  _{i=1}^{N}$. For example let
$\tau=\left(  4,3\right)  $ and $S=$ $%
\begin{array}
[c]{cccc}%
4 & 3 & 1 & \\
7 & 6 & 5 & 2
\end{array}
$ then the content vector is $\left[  1,3,0,-1,2,1,0\right]  $. There is a
representation of $\mathcal{H}_{N}\left(  t\right)  $ on $V_{\tau}$, also
denoted by $\tau$ (slight abuse of notation). The description will be given in
terms of the actions of $\left\{  T_{i}\right\}  $ on the basis elements.

\begin{definition}
\label{deftauTi}The representation $\tau$ of $\mathcal{H}_{N}\left(  t\right)
$ is defined by the action of the generators specified as follows: for $1\leq
i<N$ and $S\in\mathcal{Y}\left(  \tau\right)  $\newline1)
$\mathrm{\operatorname{row}}\left(  i,S\right)  =\mathrm{\operatorname{row}%
}\left(  i+1,S\right)  $ (implying $\operatorname{col}\left(  i,S\right)
=\operatorname{col}\left(  i+1,S\right)  +1$ and $c\left(  i,S\right)
-c\left(  i+1,S\right)  =1$) then%
\[
S\tau\left(  T_{i}\right)  =tS;
\]
\newline2) $\operatorname{col}\left(  i,S\right)  =\operatorname{col}\left(
i+1,S\right)  $ (implying $\mathrm{\operatorname{row}}\left(  i,S\right)
=\mathrm{\operatorname{row}}\left(  i+1,S\right)  +1$ and $c\left(
i,S\right)  -c\left(  i+1,S\right)  =-1$) then%
\[
S\tau\left(  T_{i}\right)  =-S;
\]
3) $\mathrm{\operatorname{row}}\left(  i,S\right)  <\mathrm{\operatorname{row}%
}\left(  i+1,S\right)  $ and $\operatorname{col}\left(  i,S\right)
>\operatorname{col}\left(  i+1,S\right)  $ then $c\left(  i,S\right)
-c\left(  i+1,S\right)  \geq2$, the tableau $S^{\left(  i\right)  }$ obtained
from $S$ by exchanging $i$ and $i+1$, is an element of $\mathcal{Y}\left(
\tau\right)  $ and%
\[
S\tau\left(  T_{i}\right)  =S^{\left(  i\right)  }+\dfrac{t-1}{1-t^{c\left(
i+1,S\right)  -c\left(  i,S\right)  }}S;
\]
4) $c\left(  i,S\right)  -c\left(  i+1,S\right)  \leq-2$, thus
$\mathrm{\operatorname{row}}\left(  i,S\right)  >\mathrm{\operatorname{row}%
}\left(  i+1,S\right)  $ and $\operatorname{col}\left(  i,S\right)
<\operatorname{col}\left(  i+1,S\right)  $ then with $b=c\left(  i,S\right)
-c\left(  i+1,S\right)  $,%
\[
S\tau\left(  T_{i}\right)  =\frac{t\left(  t^{b+1}-1\right)  \left(
t^{b-1}-1\right)  }{\left(  t^{b}-1\right)  ^{2}}S^{\left(  i\right)  }%
+\frac{t^{b}\left(  t-1\right)  }{t^{b}-1}S.
\]

\end{definition}

The formulas in (4) are consequences of those in (3) by interchanging $S$ and
$S^{\left(  i\right)  }$ and applying the relations $\left(  \tau\left(
T_{i}\right)  +I\right)  \left(  \tau\left(  T_{i}\right)  -tI\right)  =0$
(where $I$ denotes the identity operator on $V_{\tau}$). There is a partial
order on $\mathcal{Y}\left(  \tau\right)  $ related to the inversion number:%
\[
\mathrm{inv}\left(  S\right)  :=\#\left\{  \left(  i,j\right)  :1\leq i<j\leq
N,c\left(  i,S\right)  \geq c\left(  j,S\right)  +2\right\}  ,
\]
so $\mathrm{inv}\left(  S^{\left(  i\right)  }\right)  =\mathrm{inv}\left(
S\right)  -1$ in (3) above. The $\mathrm{inv}$-maximal element $S_{0}$ of
$\mathcal{Y}\left(  \tau\right)  $ has the numbers $N,N-1,\ldots,1$ entered
column-by-column and the $\mathrm{inv}$-minimal element $S_{1}$ of
$\mathcal{Y}\left(  \tau\right)  $ has the numbers $N,N-1,\ldots,1$ entered
row-by-row. The set $\mathcal{Y}\left(  \tau\right)  $ can be given the
structure of a Yang-Baxter graph, with root $S_{0}$, sink $S_{1}$ with arrows
labeled by $T_{i}$ joining $S$ to $S^{\left(  i\right)  }$ as in (3). Some
properties can be proved by induction on the inversion number. Recall
$u\left(  z\right)  =\frac{\left(  t-z\right)  \left(  1-tz\right)  }{\left(
1-z\right)  ^{2}}=u\left(  z^{-1}\right)  $.

\begin{definition}
The bilinear symmetric form $\left\langle \cdot,\cdot\right\rangle _{0}$ on
$V_{\tau}$ is defined to be the linear extension of
\[
\left\langle S,S^{\prime}\right\rangle _{0}=\delta_{S,S^{\prime}}%
\prod\limits_{i<j,~c\left(  j,S\right)  -c\left(  i,S\right)  \geq2}u\left(
t^{c\left(  i,S\right)  -c\left(  j,S\right)  }\right)  .
\]

\end{definition}

\begin{proposition}
\label{Sform}Suppose $f,g\in V_{\tau}$ then $\left\langle f\tau\left(
T_{i}\right)  ,g\right\rangle _{0}=\left\langle f,g\tau\left(  T_{i}\right)
\right\rangle _{0}$ for $1\leq i<N$. If $c\left(  i,S\right)  -c\left(
i+1,S\right)  \geq2$ for some $i,S$ then $\left\langle S^{\left(  i\right)
},S^{\left(  i\right)  }\right\rangle _{0}=u\left(  t^{c\left(  i,S\right)
-c\left(  i+1,S\right)  }\right)  \left\langle S,S\right\rangle _{0}$.
\end{proposition}

\begin{proof}
If $\mathrm{\operatorname{row}}\left(  i,S\right)  =\mathrm{\operatorname{row}%
}\left(  i+1,S\right)  $ or $\operatorname{col}\left(  i,S\right)
=\operatorname{col}\left(  i+1,S\right)  $ then $\left\langle S\tau\left(
T_{i}\right)  ,S\right\rangle _{0}=t\left\langle S,S\right\rangle
_{0}=\left\langle S,S\tau\left(  T_{i}\right)  \right\rangle _{0}$ or
$\left\langle S\tau\left(  T_{i}\right)  ,S\right\rangle _{0}=-\left\langle
S,S\right\rangle _{0}=\left\langle S,S\tau\left(  T_{i}\right)  \right\rangle
_{0}$ respectively. If $c\left(  i,S\right)  -c\left(  i+1,S\right)  \geq2$
and $b=c\left(  i+1,S\right)  -c\left(  i,S\right)  $ then $\left\langle
S^{\left(  i\right)  },S^{\left(  i\right)  }\right\rangle _{0}/\left\langle
S,S\right\rangle _{0}=u\left(  t^{-b}\right)  $ (in the product the only
difference is the term $\left(  i,i+1\right)  $, appearing in $\left\langle
S^{\left(  i\right)  },S^{\left(  i\right)  }\right\rangle _{0}$). Then
\begin{align*}
\left\langle S\tau\left(  T_{i}\right)  ,S^{\left(  i\right)  }\right\rangle
_{0}  &  =\left\langle S^{\left(  i\right)  },S^{\left(  i\right)
}\right\rangle _{0}+\dfrac{t-1}{1-t^{b}}\left\langle S^{\left(  i\right)
},S\right\rangle _{0}=\left\langle S^{\left(  i\right)  },S^{\left(  i\right)
}\right\rangle _{0},\\
\left\langle S^{\left(  i\right)  }\tau\left(  T_{i}\right)  ,S\right\rangle
_{0}  &  =\frac{t\left(  t^{b+1}-1\right)  \left(  t^{b-1}-1\right)  }{\left(
t^{b}-1\right)  ^{2}}\left\langle S,S\right\rangle _{0}+\frac{t^{b}\left(
t-1\right)  }{t^{b}-1}\left\langle S^{\left(  i\right)  },S\right\rangle
_{0}\\
&  =\frac{t\left(  t^{b+1}-1\right)  \left(  t^{b-1}-1\right)  }{\left(
t^{b}-1\right)  ^{2}}\left\langle S,S\right\rangle _{0}=u\left(  t^{b}\right)
\left\langle S,S\right\rangle _{0},
\end{align*}
thus $\left\langle S\tau\left(  T_{i}\right)  ,S^{\left(  i\right)
}\right\rangle _{0}=\left\langle S^{\left(  i\right)  }\tau\left(
T_{i}\right)  ,S\right\rangle _{0}$. These statements imply that $\left\langle
f\tau\left(  T_{i}\right)  ,g\right\rangle _{0}=\left\langle f,g\tau\left(
T_{i}\right)  \right\rangle _{0}$ for $f,g\in V_{\tau}$.
\end{proof}

Furthermore if $t>0$ then $\left\langle S,S\right\rangle _{0}\geq0$; each term
is of the form $\dfrac{\left(  t-t^{m}\right)  \left(  1-t^{m+1}\right)
}{\left(  1-t^{m}\right)  ^{2}}$ with $m\geq2$; either all parts are positive
or all are negative depending on $0<t<1$ or $t>1$ respectively (the limit as
$t\rightarrow1$ is $\frac{m^{2}-1}{m^{2}}>0$). Denote $\left\langle
f,f\right\rangle _{0}=\left\Vert f\right\Vert _{0}^{2}$ for $f\in V_{\tau}$.

There is a commutative set of Jucys-Murphy elements in $\mathcal{H}_{N}\left(
t\right)  $ which are diagonalized with respect to the basis $\mathcal{Y}%
\left(  \tau\right)  $.

\begin{definition}
Set $\phi_{N}:=1$ and $\phi_{i}:=\frac{1}{t}T_{i}\phi_{i+1}T_{i}$ for $1\leq
i<N$.
\end{definition}

\begin{proposition}
\label{eigenphi}Suppose $1\leq i\leq N$ and $S\in\mathcal{Y}\left(
\tau\right)  $ then $S\tau\left(  \phi_{i}\right)  =t^{c\left(  i,S\right)
}S$.
\end{proposition}

\begin{proof}
Arguing inductively suppose that $S\tau\left(  \phi_{i+1}\right)  =t^{c\left(
i+1,S\right)  }S$ for all $S\in\mathcal{Y}\left(  \tau\right)  $; this is
trivially true for $i=N-1$ since $c\left(  N,S\right)  =0$ and $\phi_{N}=1$.
If $\mathrm{\operatorname{row}}\left(  i,S\right)  =\mathrm{\operatorname{row}%
}\left(  i+1,S\right)  $ then $S\tau\left(  \phi_{i}\right)  =\frac{1}{t}%
S\tau\left(  T_{i}\right)  \tau\left(  \phi_{i+1}\right)  \tau\left(
T_{i}\right)  =t^{c\left(  i+1,S\right)  +1}S$ and $c\left(  i,S\right)
=c\left(  i+1,S\right)  +1$. If $\operatorname{col}\left(  i,S\right)
=\operatorname{col}\left(  i+1,S\right)  $ then $S\tau\left(  \phi_{i}\right)
=\frac{1}{t}S\tau\left(  T_{i}\right)  \tau\left(  \phi_{i+1}\right)
\tau\left(  T_{i}\right)  =\frac{1}{t}t^{c\left(  i+1,S\right)  }S$ and
$c\left(  i,S\right)  =c\left(  i+1,S\right)  -1$ (since $S\tau\left(
T_{i}\right)  =-S$). Suppose $c\left(  i,S\right)  -c\left(  i+1,S\right)
\geq2$ then the matrices $\mathcal{T},\Phi$ of $\tau\left(  T_{i}\right)
,\tau\left(  \phi_{i+1}\right)  $ respectively with respect to the basis
$\left[  S,S^{\left(  i\right)  }\right]  $ are%
\[
\mathcal{T}=%
\begin{bmatrix}
-\frac{1-t}{1-\rho} & 1\\
\frac{\left(  1-\rho t\right)  \left(  t-\rho\right)  }{\left(  1-\rho\right)
^{2}} & \frac{\rho\left(  1-t\right)  }{1-\rho}%
\end{bmatrix}
,~\Phi=%
\begin{bmatrix}
t^{c\left(  i+1,S\right)  } & 0\\
0 & t^{c\left(  i,S\right)  }%
\end{bmatrix}
,
\]
where $\rho=t^{c\left(  i+1,S\right)  -c\left(  i,S\right)  }$. A simple
calculation shows $\frac{1}{t}\mathcal{T}\Phi\mathcal{T=}%
\begin{bmatrix}
t^{c\left(  i,S\right)  } & 0\\
0 & t^{c\left(  i+1,S\right)  }%
\end{bmatrix}
$.
\end{proof}

\subsection{Polynomials and operators}

Let $\mathcal{P}_{\tau}:=\mathcal{P}\otimes V_{\tau}$. The action of
$\mathcal{H}_{N}\left(  t\right)  $ and the operators are defined as follows:
with $p\in\mathcal{P},S\in\mathcal{Y}\left(  \tau\right)  $ and $1\leq i<N$:%
\begin{align*}
\left(  p\left(  x\right)  \otimes S\right)  \boldsymbol{T}_{i}  &  :=\left(
1-t\right)  x_{i+1}\frac{p\left(  x\right)  -p\left(  x.s_{i}\right)  }%
{x_{i}-x_{i+1}}\otimes S+p\left(  x.s_{i}\right)  \otimes S\tau\left(
T_{i}\right)  ,\\
\omega &  :=T_{1}T_{2}\cdots T_{N-1},\\
\left(  p\left(  x\right)  \otimes S\right)  \boldsymbol{w}  &  :=p\left(
qx_{N},x_{1},\ldots,x_{N-1}\right)  \otimes S\tau\left(  \omega\right) \\
\boldsymbol{\xi}_{i}  &  :=t^{i-N}\boldsymbol{T}_{i-1}^{-1}\cdots
\boldsymbol{T}_{1}^{-1}\boldsymbol{wT}_{N-1}\cdots\boldsymbol{T}_{i}\\
\mathcal{D}_{N}  &  :=\left(  1-\boldsymbol{\xi}_{N}\right)  /x_{N}%
,~\mathcal{D}_{i}:=\frac{1}{t}\boldsymbol{T}_{i}\mathcal{D}_{i+1}%
\boldsymbol{T}_{i}.
\end{align*}
By the braid relations
\begin{align*}
T_{i+1}\omega &  =T_{1}\cdots T_{i+1}T_{i}T_{i+1}T_{i+2}\cdots T_{N-1}\\
&  =T_{1}\cdots T_{i}T_{i+1}T_{i}T_{i+2}\cdots T_{N-1}=\omega T_{i}.
\end{align*}
for $1\leq i<N-1$. It follows that $\boldsymbol{T}_{i+1}\boldsymbol{w}%
=\boldsymbol{w}\boldsymbol{T}_{i}$ acting on $\mathcal{P}_{\tau}$. The
operators $\left\{  \boldsymbol{\xi}_{i}\right\}  $ mutually commute and the
simultaneous polynomial eigenfunctions are the vector-valued (nonsymmetric)
Macdonald polynomials. The factor $t^{i-N}$ in $\boldsymbol{\xi}_{i}$ appears
to differ from the scalar case, but if $\tau=\left(  N\right)  $, the trivial
representation, then $S\tau\left(  T_{i}\right)  =tS$ (the unique RSYT of
shape $\left(  N\right)  $) and $S\tau\left(  \omega\right)  =t^{N-1}S$, and
thus $\boldsymbol{\xi}_{i}$ coincides with (\ref{defxi0}). The operator
$\boldsymbol{\xi}_{i}$ acting on constants coincides with $I\otimes\tau\left(
\phi_{i}\right)  $:%
\begin{align*}
\left(  1\otimes S\right)  \boldsymbol{\xi}_{i}  &  =t^{i-N}\otimes
S\tau\left(  T_{i-1}^{-1}\cdots T_{1}^{-1}T_{1}T_{2}\cdots T_{N-1}%
T_{N-1}\cdots T_{i}\right) \\
&  =t^{i-N}\otimes S\tau\left(  T_{i}\cdots T_{N-1}T_{N-1}\cdots T_{i}\right)
=1\otimes S\tau\left(  \phi_{i}\right)  =t^{c\left(  i,S\right)  }\left(
1\otimes S\right)  .
\end{align*}
For each $\alpha\in\mathbb{N}_{0}^{N}$ and $S\in\mathcal{Y}\left(
\tau\right)  $ there is an $\left\{  \boldsymbol{\xi}_{i}\right\}  $
eigenfunction%
\[
M_{\alpha,S}\left(  x\right)  =\eta\left(  \alpha,S\right)  x^{\alpha}\otimes
S\tau\left(  R_{\alpha}\right)  +\sum_{\alpha\vartriangleright\beta}x^{\beta
}\otimes B_{\alpha,\beta}\left(  q,t\right)
\]
where $\eta\left(  \alpha,S\right)  =q^{a}t^{b}$ with $a,b\in\mathbb{N}_{0}$
and $R_{\alpha},B_{\alpha,\beta}\left(  q,t\right)  \in\mathcal{H}_{N}\left(
t\right)  .$ Furthermore $R_{\alpha}$ is an analog of $r_{\alpha}$ (see
\cite[p.9]{DL2012}); if $\alpha\in\mathbb{N}_{0}^{N,+}$ then $R_{\alpha}=I$,
and if $\alpha_{i}<\alpha_{i+1}$ then $R_{\alpha.s_{i}}=R_{\alpha}T_{i}$
(there is a definition of $R_{\alpha}$ below). Furthermore%
\begin{align*}
M_{\alpha,S}\boldsymbol{\xi}_{i}  &  =\zeta_{\alpha,S}\left(  i\right)
M_{\alpha,S},1\leq i\leq N,\\
\zeta_{\alpha,S}\left(  i\right)   &  =q^{\alpha_{i}}t^{c\left(  r_{\alpha
}\left(  i\right)  ,S\right)  }.
\end{align*}
These polynomials are produced with the Yang-Baxter graph. The typical node
(labeled by $\left(  \alpha,S\right)  $) is%
\[
\left(  \alpha,S,\zeta_{\alpha,S},R_{\alpha},M_{\alpha,S}\right)
\]
and the root is $\left(  \boldsymbol{0},S_{0},\left[  t^{c\left(
i,S_{0}\right)  }\right]  _{i=1}^{N},I,1\otimes S_{0}\right)  $.

There are steps:

\begin{itemize}
\item if $\alpha_{i}<\alpha_{i+1}$ there is a step labeled $s_{i}$%
\begin{align}
\left(  \alpha,S,\zeta_{\alpha,S},R_{\alpha},M_{\alpha,S}\right)   &
\rightarrow\left(  \alpha.s_{i},S,\zeta_{\alpha.s_{i},S},R_{\alpha.s_{i}%
},M_{\alpha.s_{i},S}\right)  ,\nonumber\\
M_{\alpha.s_{i},S}  &  =M_{\alpha,S}\boldsymbol{T}_{i}+\frac{t-1}%
{\zeta_{\alpha,S}\left(  i+1\right)  /\zeta_{\alpha,S}\left(  i\right)
-1}M_{\alpha,S}.\label{MT>i}\\
R_{\alpha.s_{i}}  &  =R_{\alpha}T_{i},~\eta\left(  \alpha.s_{i},S\right)
=\eta\left(  \alpha,S\right)  ;\nonumber
\end{align}
(note that $\left(  x_{i}^{\alpha_{i}}x_{i+1}^{\alpha_{i+1}}\otimes
SR_{\alpha}\right)  \boldsymbol{T}_{i}=x_{i}^{\alpha_{i+1}}x_{i+1}^{\alpha
_{i}}\otimes S\tau\left(  R_{\alpha}T_{i}\right)  +\cdots$.)

\item if $\alpha_{i}=\alpha_{i+1}$, $j=r_{\alpha}\left(  i\right)  $ (thus
$j+1=r_{\alpha}\left(  i+1\right)  $, and $R_{\alpha}T_{i}=T_{j}R_{\alpha}$,
(see \cite[Lemma 2.14]{DL2012}) and $c\left(  j,S\right)  -c\left(
j+1,S\right)  \geq2$ there is a step%
\begin{align}
\left(  \alpha,S,\zeta_{\alpha,S},R_{\alpha},M_{\alpha,S}\right)   &
\rightarrow\left(  \alpha,S^{\left(  j\right)  },\left(  \zeta_{\alpha
,S}\right)  .s_{i},R_{\alpha},M_{\alpha,S^{\left(  j\right)  }}\right)
,\nonumber\\
M_{\alpha,S^{\left(  j\right)  }}  &  =M_{\alpha,S}\boldsymbol{T}_{i}%
+\frac{t-1}{\zeta_{\alpha,S}\left(  i+1\right)  /\zeta_{\alpha,S}\left(
i\right)  -1}M_{\alpha,S},\label{MT=i}\\
\frac{\zeta_{\alpha,S}\left(  i+1\right)  }{\zeta_{\alpha,S}\left(  i\right)
}  &  =t^{c\left(  j+1,S\right)  -c\left(  j,S\right)  },\eta\left(
\alpha,S^{\left(  j\right)  }\right)  =\eta\left(  \alpha,S\right)  ;\nonumber
\end{align}

\end{itemize}

For these formulas to be valid it is required that the denominators
$\zeta_{\alpha,S}\left(  i+1\right)  /\zeta_{\alpha,S}\left(  i\right)  -1$ do
not vanish, that is, $q^{\lambda_{i+1}-\lambda_{i}}t^{c\left(  r_{\alpha
}\left(  i+1\right)  ,S\right)  -c\left(  r_{\alpha}\left(  i\right)
,S\right)  }\neq1$. From the bound $\left\vert c\left(  j,S\right)  -c\left(
j^{\prime},S\right)  \right\vert \leq\tau_{1}+\ell\left(  \tau\right)  -2$ we
obtain the necessary condition $q^{q}t^{b}\neq1$ for $a\geq0$ and $\left\vert
b\right\vert \leq\tau_{1}+\ell\left(  \tau\right)  -2$ These conditions are
satisfied in the region of positivity described in Proposition \ref{regpos2}.

The other possibilities for the action of $\boldsymbol{T}_{i}$ are:

\begin{itemize}
\item if $\alpha_{i}>\alpha_{i+1}$ set $\rho:=\zeta_{\alpha,S}\left(
i\right)  /\zeta_{\alpha,S}\left(  i+1\right)  $ then%
\begin{equation}
M_{\alpha,S}\boldsymbol{T}_{i}=\frac{\left(  1-t\rho\right)  \left(
t-\rho\right)  }{\left(  1-\rho\right)  ^{2}}M_{\alpha.s_{i},S}+\frac
{\rho\left(  1-t\right)  }{\left(  1-\rho\right)  }M_{\alpha,S}; \label{MT<i}%
\end{equation}

\item if $\alpha_{i}=\alpha_{i+1}$ and $j=r_{\alpha}\left(  i\right)
,~c\left(  j,S\right)  -c\left(  j+1,S\right)  \leq2,~\rho=t^{c\left(
j,S\right)  -c\left(  j+1,S\right)  }$ then%
\begin{equation}
M_{\alpha,S}\boldsymbol{T}_{i}=\frac{\left(  1-t\rho\right)  \left(
t-\rho\right)  }{\left(  1-\rho\right)  ^{2}}M_{\alpha,S^{\left(  j\right)  }%
}+\frac{\rho\left(  1-t\right)  }{\left(  1-\rho\right)  }M_{\alpha,S};
\label{MT=2}%
\end{equation}

\item if $\alpha_{i}=\alpha_{i+1}$ and $j=r_{\alpha}\left(  i\right)
,\mathrm{\operatorname{row}}\left(  j,S\right)  =\mathrm{\operatorname{row}%
}\left(  j+1,S\right)  $ then $M_{\alpha,S}\boldsymbol{T}_{i}=tM_{\alpha,S}$;

\item if $\alpha_{i}=\alpha_{i+1}$ and $j=r_{\alpha}\left(  i\right)
,\operatorname{col}\left(  j,S\right)  =\operatorname{col}\left(
j+1,S\right)  $ then $M_{\alpha,S}\boldsymbol{T}_{i}=-M_{\alpha,S}$ .
\end{itemize}

The degree-raising operation, namely, the affine step, takes $\alpha$ to
$\alpha\Phi:=\left(  \alpha_{2},\alpha_{3},\ldots,\alpha_{N},\alpha
_{1}+1\right)  $:
\begin{align*}
&  \left(  \alpha,S,\zeta_{\alpha,S},R_{\alpha},M_{\alpha,S}\right)
\overset{\Phi}{\rightarrow}\left(  \alpha\Phi,S,\zeta_{\alpha\Phi,S}%
,R_{\alpha\Phi},M_{\alpha\Phi,S}\right)  ,\\
M_{\alpha\Phi,S}\left(  x\right)   &  =x_{N}M_{\alpha\Phi,S}\boldsymbol{w},\\
\alpha\Phi &  =\left(  \alpha_{2},\alpha_{3},\ldots,\alpha_{N},\alpha
_{1}+1\right)  ,\\
\zeta_{\alpha\Phi,S}  &  =\left(  \zeta_{\alpha,S}\left(  2\right)
,\ldots,\zeta_{\alpha,S}\left(  N\right)  ,q\zeta_{\alpha,S}\left(  1\right)
\right)  .
\end{align*}

The inversion number $\mathrm{inv}\left(  \alpha\right)  $ of $\alpha
\in\mathbb{N}_{0}^{N}$ is the length of the shortest product $g=s_{i_{1}%
}s_{i_{2}}\cdots s_{i_{m}}$ such that $\alpha.g=\alpha^{+}$. From this and the
Yang-Baxter graph we deduce that the series of steps $s_{i_{1}},s_{i_{2}%
},\cdots,s_{i_{m}}$ lead from $M_{\alpha,S}$ to $M_{\alpha^{+},S}$ and
$R_{\alpha}T_{i_{i}}T_{i_{2}}\cdots T_{i_{m}}=R_{\alpha^{+}}=I$.

\begin{definition}
Suppose $\alpha\in\mathbb{N}_{0}^{N}$ then $R_{\alpha}:=\left(  T_{i_{i}%
}T_{i_{2}}\cdots T_{i_{m}}\right)  ^{-1}$ where $\alpha.s_{i_{1}}s_{i_{2}%
}\cdots s_{i_{m}}=\alpha^{+}$ and $m=\mathrm{inv}\left(  \alpha\right)  $.
\end{definition}

There may be different products $\alpha.s_{j_{1}}s_{j_{2}}\cdots s_{ij}%
=\alpha^{+}$ of length $\mathrm{inv}\left(  \alpha\right)  $ but they all give
the same value of $R_{\alpha}$ by the braid relations. It is shown in
\cite[p.10, (2.15)]{DL2012} that $R_{\alpha}\omega=t^{N-m}\phi_{m}%
R_{\alpha\Phi}$ with $m=r_{\alpha}\left(  1\right)  $.

\subsection{The bilinear symmetric form}

We will define a symmetric bilinear form $\left\langle \cdot,\cdot
\right\rangle $ on $\mathcal{P}_{\tau}$ satisfying certain postulates; using
the same logical outline as in Section \ref{SMD}; first we derive consequences
from these, then state the definition and show the desired properties apply.

The \textbf{hypotheses} (BF2) for the symmetric bilinear form $\left\langle
\cdot,\cdot\right\rangle $ on $\mathcal{P}_{\tau}$, with $\boldsymbol{w}%
^{\ast}:=\boldsymbol{T}_{N-1}^{-1}\cdots\boldsymbol{T}_{1}^{-1}\boldsymbol{w}%
\boldsymbol{T}_{N-1}\cdots\boldsymbol{T}_{1}$, are (for $f,g\in\mathcal{P}%
_{\tau},~S,S^{\prime}\in\mathcal{Y}\left(  \tau\right)  ,~1\leq i<N$)%
\begin{subequations}
\begin{align}
\left\langle 1\otimes S,1\otimes S^{\prime}\right\rangle  &  =\left\langle
S,S^{\prime}\right\rangle _{0},\label{BF2.1}\\
\left\langle f\boldsymbol{T}_{i},g\right\rangle  &  =\left\langle
f,g\boldsymbol{T}_{i}\right\rangle ,\label{BF2.2}\\
\left\langle f\boldsymbol{\xi}_{N},g\right\rangle  &  =\left\langle
f,g\boldsymbol{\xi}_{N}\right\rangle .\label{BF2.3}\\
\left\langle f\mathcal{D}_{N},g\right\rangle  &  =\left(  1-q\right)
\left\langle f,x_{N}\left(  g\boldsymbol{w}^{\ast}\boldsymbol{w}\right)
\right\rangle \label{BF2.4}%
\end{align}
Properties (\ref{BF2.2}) and (\ref{BF2.3}) imply $\left\langle
f\boldsymbol{\xi}_{i},g\right\rangle =\left\langle f,g\boldsymbol{\xi}%
_{i}\right\rangle $ for each $i$ and thus the $M_{\alpha,S}$'s are mutually
orthogonal. As in the scalar case $\left\langle f\boldsymbol{w},g\right\rangle
=\left\langle f,g\boldsymbol{w}^{\ast}\right\rangle $. If $f,g$ are
homogeneous of different degrees then $\left\langle f,g\right\rangle =0$. As
before denote $\left\langle f,f\right\rangle =\left\Vert f\right\Vert ^{2}$.
First we will show that these hypotheses determine the from uniquely when
$q,t\neq0,1$ without recourse to the Macdonald polynomials. We use the
commutation relationships $\left(  x_{i+1}f\right)  \boldsymbol{T}_{i}%
=x_{i}\left(  f\boldsymbol{T}_{i}\right)  +\left(  t-1\right)  x_{i+1}f$ and
$\left(  x_{j}f\right)  \boldsymbol{T}_{i}=x_{j}\left(  f\boldsymbol{T}%
_{i}\right)  $ for $f\in\mathcal{P}_{\tau}$ and $j\neq i,i+1$ (a simple direct computation).
\end{subequations}
\begin{proposition}
\label{fDxg}Suppose (BF2) holds then for $1\leq i\leq j\leq N$ and
$q,t\neq0,1$ there are operators $A_{i,j},B_{i,j}$,on $\mathcal{P}_{\tau}$
preserving degree of homogeneity such that $A_{i,i}$ and $B_{i,i}$ are
invertible and for $f,g\in\mathcal{P}_{\tau}$%
\begin{align*}
\left\langle f\mathcal{D}_{i},g\right\rangle  &  =\sum_{j=i}^{N}\left\langle
f,x_{j}\left(  gA_{i,j}\right)  \right\rangle \\
\left\langle f,x_{i}g\right\rangle  &  =\sum_{j=i}^{N}\left\langle
f\mathcal{D}_{j},gB_{i,j}\right\rangle .
\end{align*}

\end{proposition}

\begin{proof}
Suppose $i=N$ then $A_{N,N}=\left(  1-q\right)  \boldsymbol{w}^{\ast
}\boldsymbol{w}$ and $B_{N,N}=A_{N,N}^{-1}$ by (\ref{BF2.4}). Arguing by
induction suppose the statement is true for $k+1\leq i\leq N$. Then for any
$f,g\in\mathcal{P}_{\tau}$%
\begin{align*}
\left\langle f\mathcal{D}_{k},g\right\rangle  &  =\frac{1}{t}\left\langle
f\boldsymbol{T}_{k}\mathcal{D}_{k+1}\boldsymbol{T}_{k},g\right\rangle
=\frac{1}{t}\left\langle f\boldsymbol{T}_{k}\mathcal{D}_{k+1},g\boldsymbol{T}%
_{k}\right\rangle \\
&  =\frac{1}{t}\sum_{j=k+1}^{N}\left\langle f\boldsymbol{T}_{k},x_{j}\left(
g\boldsymbol{T}_{k}A_{k+1,j}\right)  \right\rangle =\frac{1}{t}\sum
_{j=k+1}^{N}\left\langle f,\left\{  x_{j}\left(  g\boldsymbol{T}_{k}%
A_{k+1,j}\right)  \right\}  \boldsymbol{T}_{k}\right\rangle .
\end{align*}
Then
\begin{align*}
\left\{  x_{k+1}\left(  g\boldsymbol{T}_{k}A_{k+1,k+1}\right)  \right\}
\boldsymbol{T}_{k}  &  =x_{k}\left(  g\boldsymbol{T}_{k}A_{k+1,k+1}%
\boldsymbol{T}_{k}\right)  +\left(  t-1\right)  x_{k+1}\left(  g\boldsymbol{T}%
_{k}A_{k+1,k+1}\right)  ,\\
\left\{  x_{j}\left(  g\boldsymbol{T}_{k}A_{k+1,j}\right)  \right\}
\boldsymbol{T}_{k}  &  =x_{j}\left(  g\boldsymbol{T}_{k}A_{k+1,j}%
\boldsymbol{T}_{k}\right)  .
\end{align*}
Thus set $A_{k,k}:=\frac{1}{t}\boldsymbol{T}_{k}A_{k+1,k+1}\boldsymbol{T}_{k}%
$, $A_{k,k+1}:=\frac{t-1}{t}\boldsymbol{T}_{k}A_{k+1,k+1}$ and $A_{k,j}%
:=\frac{1}{t}\boldsymbol{T}_{k}A_{k+1,j}\boldsymbol{T}_{k}$ for $j>k+1$. Next
\[
\left\langle f,x_{k}\left(  gA_{k,k}\right)  \right\rangle =\left\langle
f\mathcal{D}_{k},g\right\rangle -\sum_{j=k+1}^{N}\left\langle f,x_{j}\left(
gA_{k,j}\right)  \right\rangle .
\]
Replace $g$ by $gA_{k,k}^{-1}$ and use the inductive hypothesis:%
\begin{align*}
\left\langle f,x_{k}g\right\rangle  &  =\left\langle f\mathcal{D}_{k}%
,gA_{k,k}^{-1}\right\rangle +\sum_{m=k+1}^{N}\left\langle f\mathcal{D}%
_{m},gB_{k,m}\right\rangle ,\\
B_{k,m}  &  :=-\sum_{j=k+1}^{m}A_{k,k}^{-1}A_{k,j}B_{j,m},B_{k,k}%
:=A_{k,k}^{-1}.
\end{align*}
This completes the induction.
\end{proof}

This is an uniqueness result because it shows that inner products involving
any $f\in\mathcal{P}_{n}\otimes V_{\tau}$ are determined by inner products
with $\left\{  f\mathcal{D}_{i}:1\leq i\leq N\right\}  \subset\mathcal{P}%
_{n-1}\otimes V_{\tau}$. However the result does not prove existence or
symmetry. A closer look at the formulas shows that $\left(  1-q\right)
^{\left\vert \alpha\right\vert }\left\langle x^{\alpha}\otimes S,x^{\beta
}\otimes S^{\prime}\right\rangle $ is a Laurent polynomial in $q,t$ (a sum of
$q^{q}t^{b}$ with $a,b\in\mathbb{Z}$) for any $\alpha,\beta\in\mathbb{N}%
_{0}^{N}$, $S,S^{\prime}\in\mathcal{Y}\left(  \tau\right)  $.

Recall $u\left(  z\right)  :=\frac{\left(  1-tz\right)  \left(  t-z\right)
}{\left(  1-z\right)  ^{2}}$.

\begin{lemma}
Suppose (BF2) hold and suppose $\left(  \alpha,S\right)  $ satisfies
$\alpha_{i}<\alpha_{i+1}$ then with $\rho=\zeta_{\alpha,S}\left(  i+1\right)
/\zeta_{\alpha,S}\left(  i\right)  $
\[
\left\Vert M_{\alpha.s_{i},S}\right\Vert ^{2}=u\left(  \rho\right)  \left\Vert
M_{\alpha,S}\right\Vert ^{2}.
\]

\end{lemma}

\begin{proof}
From (\ref{MT>i}) and (\ref{MT<i})%
\begin{align*}
M_{\alpha,S}\boldsymbol{T}_{i}  &  =-\frac{1-t}{1-\rho}M_{\alpha,S}%
+M_{\alpha.s_{i},S},\\
M_{\alpha.s_{i},S}\boldsymbol{T}_{i}  &  =\frac{\left(  1-t\rho\right)
\left(  t-\rho\right)  }{\left(  1-\rho\right)  ^{2}}M_{\alpha,S}+\frac
{\rho\left(  1-t\right)  }{\left(  1-\rho\right)  }M_{\alpha.s_{i},S}.
\end{align*}
Take the inner product of the first equation with $p$ and use $\left\langle
M_{\alpha,S},M_{\alpha.s_{i},S}\right\rangle =0$, then take the inner product
of the second equation with $M_{\alpha,S}$ and again use $\left\langle
M_{\alpha,S},M_{\alpha.s_{i},S}\right\rangle =0$ to obtain%
\begin{align*}
\left\langle M_{\alpha,S}\boldsymbol{T}_{i},M_{\alpha.s_{i},S}\right\rangle
&  =\left\Vert M_{\alpha.s_{i},S}\right\Vert ^{2},\\
\left\langle M_{\alpha,S},M_{\alpha.s_{i},S}\boldsymbol{T}_{i}\right\rangle
&  =\frac{\left(  1-t\rho\right)  \left(  t-\rho\right)  }{\left(
1-\rho\right)  ^{2}}\left\Vert M_{\alpha,S}\right\Vert ^{2}.
\end{align*}
The hypothesis $\left\langle M_{\alpha,S}\boldsymbol{T}_{i},p\right\rangle
=\left\langle M_{\alpha,S},M_{\alpha.s_{i},S}\boldsymbol{T}_{i}\right\rangle $
completes the proof.
\end{proof}

\begin{lemma}
Suppose (BF2) hold and suppose $\left(  \alpha,S\right)  $ satisfies
$\alpha_{i}=\alpha_{i+1},j=r_{\alpha}\left(  i\right)  ,c\left(  j,S\right)
-c\left(  j+1,S\right)  \geq2$ then with $\rho=\zeta_{\alpha,S}\left(
i+1\right)  /\zeta_{\alpha,S}\left(  i\right)  =t^{c\left(  j+1,S\right)
-c\left(  j,S\right)  }$%
\[
\left\Vert M_{\alpha,S^{\left(  j\right)  }}\right\Vert ^{2}=u\left(
\rho\right)  \left\Vert M_{\alpha,S}\right\Vert ^{2}=\frac{\left\Vert
S^{\left(  j\right)  }\right\Vert _{0}^{2}}{\left\Vert S\right\Vert _{0}^{2}%
}\left\Vert M_{\alpha,S}\right\Vert ^{2}.
\]

\end{lemma}

\begin{proof}
Using the same argument as in the previous lemma on formulas (\ref{MT=i}) and
(\ref{MT=2}) shows $\left\Vert M_{\alpha,S^{\left(  j\right)  }}\right\Vert
^{2}=u\left(  \rho\right)  \left\Vert M_{\alpha,S}\right\Vert ^{2}$.
Proposition \ref{Sform} asserted that $u\left(  \rho\right)  =\left\Vert
S^{\left(  j\right)  }\right\Vert _{0}^{2}/\left\Vert S\right\Vert _{0}^{2}$ .
\end{proof}

\begin{definition}
For $\alpha\in\mathbb{N}_{0}^{N},S\in\mathcal{Y}\left(  \tau\right)  $ let%
\[
\mathcal{E}\left(  \alpha,S\right)  :=\prod\limits_{1\leq i<j\leq N,\alpha
_{i}<\alpha_{j}}u\left(  q^{\alpha_{j}-\alpha_{i}}t^{c\left(  r_{\alpha
}\left(  j\right)  ,S\right)  -c\left(  r_{\alpha}\left(  i\right)  ,S\right)
}\right)  .
\]

\end{definition}

There are $\mathrm{inv}\left(  \alpha\right)  $ terms in $\mathcal{E}\left(
\alpha,S\right)  $.

\begin{lemma}
\label{Mdiff2}Suppose $\alpha\in\mathbb{N}_{0}^{N},S\in\mathcal{Y}\left(
\tau\right)  $ then $M_{\alpha\Phi,S}\mathcal{D}_{N}=\left(  1-q\zeta
_{\alpha,S}\left(  1\right)  \right)  M_{\alpha,S}\boldsymbol{w}$.
\end{lemma}

\begin{proof}
By definition%
\begin{align*}
M_{\alpha\Phi,S}\mathcal{D}_{N}  &  =\left(  1/x_{N}\right)  M_{\alpha\Phi
,S}\left(  I-\boldsymbol{\xi}_{N}\right)  =\left(  1/x_{N}\right)  \left(
1-\zeta_{\alpha\Phi,S}\left(  N\right)  \right)  M_{\alpha\Phi,S}\\
&  =\left(  1-q\zeta_{\alpha,S}\left(  1\right)  \right)  M_{\alpha
,S}\boldsymbol{w}.
\end{align*}

\end{proof}

The following is proved exactly like Propositions \ref{e-norm1} and
\ref{phinorm1}.

\begin{proposition}
\label{ephinorm}Suppose (BF2) holds and $\alpha\in\mathbb{N}_{0}^{N}%
,S\in\mathcal{Y}\left(  \tau\right)  $ then
\begin{align*}
\left\Vert M_{\alpha^{+},S}\right\Vert ^{2}  &  =\mathcal{E}\left(
\alpha,S\right)  \left\Vert M_{\alpha,S}\right\Vert ^{2},\\
\left\Vert M_{\alpha,\Phi,S}\right\Vert ^{2}  &  =\frac{1-q^{\alpha_{1}%
+1}t^{c\left(  r_{\alpha}\left(  1\right)  ,S\right)  }}{1-q}\left\Vert
M_{\alpha,S}\right\Vert ^{2}.
\end{align*}

\end{proposition}

The intention here is to find the explicit formula for $\left\Vert
M_{\alpha,S}\right\Vert ^{2}$ implied by (BF2) and then prove that as a
definition it satisfies (BF2). We use the same inductive scheme as in Section
\ref{SMD}.

Suppose (BF2) holds and $\lambda\in\mathbb{N}_{0}^{N,+},S\in\mathcal{Y}\left(
\tau\right)  $ and $\lambda_{m}>0=\lambda_{m+1}$ then set%
\begin{align}
\alpha &  :=\left(  \lambda_{1},\ldots,\lambda_{m-1},0,\ldots0,\lambda
_{m}\right)  ,r_{\alpha}=\left(  1,\ldots,m-1,m+1,\ldots,N,m\right)
,\label{abgamma}\\
\beta &  :=\left(  \lambda_{m}-1,\lambda_{1},\ldots,\lambda_{m-1}%
,0,\ldots\right)  ,r_{\beta}=\left(  m,1,\ldots,m-1,m+1,\ldots,N\right)
,\nonumber\\
\gamma &  :=\left(  \lambda_{1},\ldots,\lambda_{m-1},\lambda_{m}%
-1,0,\ldots\right)  =\beta^{+}.\nonumber
\end{align}
Thus $\left\Vert M_{\lambda,S}\right\Vert ^{2}=\mathcal{E}\left(
\alpha,S\right)  \left\Vert M_{\alpha,S}\right\Vert ^{2}$ and $\left\Vert
M_{\beta,S}\right\Vert ^{2}=\mathcal{E}\left(  \beta,S\right)  ^{-1}\left\Vert
M_{\gamma,S}\right\Vert ^{2}$; by Proposition \ref{ephinorm} $\left\Vert
M_{\alpha,S}\right\Vert ^{2}=\dfrac{1-q^{\lambda_{m}}t^{c\left(  m,S\right)
}}{1-q}\left\Vert M_{\beta,S}\right\Vert ^{2}$. Also $\alpha.\left(
s_{N-1}s_{N-2}\cdots s_{m}\right)  =\lambda$ and $\beta.\left(  s_{1}%
s_{2}\cdots s_{m-1}\right)  =\gamma$ thus $R_{\alpha}=T_{m}^{-1}\cdots
T_{N-1}^{-1}$ and $R_{\beta}=T_{m-1}^{-1}\cdots T_{1}^{-1}$. The leading term
of $M_{\beta,S}$ is $\eta\left(  \beta,S\right)  x^{\beta}\otimes S\tau\left(
R_{\beta}\right)  $, so the leading term of $M_{\gamma,S}$ is $\eta\left(
\beta,S\right)  x^{\gamma}\otimes S$ (and $\eta\left(  \gamma,S\right)
=\eta\left(  \beta,S\right)  $).

Apply $\boldsymbol{w}$ to $M_{\beta,S}$ then%
\begin{align*}
x_{N}\left(  \left(  x^{\beta}\right)  w\right)  S\tau\left(  R_{\beta}%
\omega\right)   &  =q^{\beta_{1}}x^{\alpha}\otimes S\tau\left(  \left(
T_{m-1}^{-1}\cdots T_{1}^{-1}\right)  T_{1}\cdots T_{N-1}\right) \\
&  =q^{\beta_{1}}x^{\alpha}\otimes S\tau\left(  T_{m}\cdots T_{N-1}\right)  ,
\end{align*}
and%
\begin{align*}
S\tau\left(  T_{m}\cdots T_{N-1}\right)   &  =S\tau\left(  \left(  T_{m}\cdots
T_{N-1}\right)  \left(  T_{N-1}\cdots T_{m}\right)  R_{\alpha}\right) \\
&  =t^{N-m}S\tau\left(  \phi_{m}R_{\alpha}\right)  =t^{N-m+c\left(
m,S\right)  }S\tau\left(  R_{\alpha}\right)  ,
\end{align*}
thus
\begin{align*}
\eta\left(  \alpha,S\right)   &  =q^{\lambda_{m}-1}t^{N-m+c\left(  m,S\right)
}\eta\left(  \beta,S\right)  ,\\
\eta\left(  \lambda,S\right)   &  =\eta\left(  \alpha,S\right)  =q^{\lambda
_{m}-1}t^{N-m+c\left(  m,S\right)  }\eta\left(  \gamma,S\right)  .
\end{align*}
Compute%
\begin{align*}
\mathcal{E}\left(  \alpha,S\right)   &  =\prod\limits_{j=m+1}^{N}u\left(
q^{\lambda_{m}}t^{c\left(  m,S\right)  -c\left(  j,S\right)  }\right)  ,\\
\mathcal{E}\left(  \beta,S\right)   &  =\prod\limits_{i=1}^{m-1}u\left(
q^{\lambda_{i}-\lambda_{m}+1}t^{c\left(  i,S\right)  -c\left(  m,S\right)
}\right)  .
\end{align*}

The argument also shows that $\eta\left(  \lambda,S\right)  =q^{\Sigma
_{1}\left(  \lambda\right)  }t^{\Sigma_{2}\left(  \lambda,S\right)  }$ where
$\Sigma_{1}\left(  \lambda\right)  :=\frac{1}{2}\sum_{i=1}^{N}\lambda
_{i}\left(  \lambda_{i}-1\right)  $ and $\Sigma_{2}\left(  \lambda,S\right)
=\sum_{i=1}^{N}\lambda_{i}\left(  N-i+c\left(  i,S\right)  \right)  $. Recall
$k\left(  \lambda\right)  =\sum_{i=1}^{N}\left(  N-2i+1\right)  \lambda_{i}$
for $\lambda\in\mathbb{N}_{0}^{N,+}$.

\begin{theorem}
Suppose (BF2) holds, $\lambda\in\mathbb{N}_{0}^{N,+}$ and $S\in\mathcal{Y}%
\left(  \tau\right)  $ then%
\begin{align}
\left\Vert M_{\lambda,S}\right\Vert ^{2}  &  =t^{k\left(  \lambda\right)
}\left\Vert S\right\Vert _{0}^{2}\ \left(  1-q\right)  ^{-\left\vert
\lambda\right\vert }\prod\limits_{i=1}^{N}\left(  qt^{c\left(  i,S\right)
};q\right)  _{\lambda_{i}}\label{normdef2}\\
&  \times\prod\limits_{1\leq i<j\leq N}\frac{\left(  qt^{c\left(  i,S\right)
-c\left(  j,S\right)  -1};q\right)  _{\lambda_{i}-\lambda_{j}}\left(
qt^{c\left(  i,S\right)  -c\left(  j,S\right)  +1};q\right)  _{\lambda
_{i}-\lambda_{j}}}{\left(  qt^{c\left(  i,S\right)  -c\left(  j,S\right)
};q\right)  _{\lambda_{i}-\lambda_{j}}^{2}}.\nonumber
\end{align}

\end{theorem}

\begin{proof}
Denote the $\left(  i,j\right)  $-product by $\Pi_{\lambda}$. Suppose
$\lambda_{m}>0=\lambda_{m+1}$ and $\gamma=\left(  \lambda_{1},\ldots
,\lambda_{m-1},\lambda_{m}-1,0,\ldots\right)  $.with $\alpha,\beta$ as in
(\ref{abgamma}). Then%
\begin{align*}
\frac{\Pi_{\lambda}}{\Pi_{\gamma}}  &  =\prod\limits_{i=1}^{m-1}\frac{\left(
1-q^{\lambda_{i}-\lambda_{m}+1}t^{c\left(  i,S\right)  -c\left(  m,S\right)
}\right)  ^{2}}{\left(  1-q^{\lambda_{i}-\lambda_{m}+1}t^{c\left(  i,S\right)
-c\left(  m,S\right)  -1}\right)  \left(  1-q^{\lambda_{i}-\lambda_{m}%
+1}t^{c\left(  i,S\right)  -c\left(  m,S\right)  +1}\right)  }\\
&  \times\prod\limits_{j=m+1}^{N}\frac{\left(  1-q^{\lambda_{m}}t^{c\left(
m,S\right)  -c\left(  j,S\right)  -1}\right)  \left(  1-q^{\lambda_{m}%
}t^{c\left(  m,S\right)  -c\left(  j,S\right)  +1}\right)  }{\left(
1-q^{\lambda_{m}}t^{c\left(  m,S\right)  -c\left(  j,S\right)  }\right)  ^{2}%
}\\
&  =t^{2m-1-N}\prod\limits_{i=1}^{m-1}u\left(  q^{\lambda_{i}-\lambda_{m}%
+1}t^{c\left(  i,S\right)  -c\left(  m,S\right)  }\right)  ^{-1}%
\prod\limits_{j=m+1}^{N}u\left(  q^{\lambda_{m}}t^{c\left(  m,S\right)
-c\left(  j,S\right)  }\right) \\
&  =t^{2m-1-N}\mathcal{E}\left(  \alpha,S\right)  /\mathcal{E}\left(
\beta,S\right)  .
\end{align*}
Also $\prod\limits_{i=1}^{N}\left(  qt^{c\left(  i,S\right)  };q\right)
_{\lambda_{i}}/\prod\limits_{i=1}^{N}\left(  qt^{c\left(  i,S\right)
};q\right)  _{\gamma_{i}}=1-q^{\lambda_{m}}t^{c\left(  m,S\right)  }$. The
formula satisfies the relation $\left\Vert M_{\lambda,S}\right\Vert
^{2}=\dfrac{1-q^{\lambda_{m}}t^{c\left(  m,S\right)  }}{1-q}\dfrac
{\mathcal{E}\left(  \alpha,S\right)  }{\mathcal{E}\left(  \beta,S\right)
}\left\Vert M_{\gamma,S}\right\Vert ^{2}$ and is valid at $\lambda
=\boldsymbol{0}$ since $M_{\boldsymbol{0},S}=1\otimes S$ and $\left\Vert
1\otimes S\right\Vert ^{2}=\left\Vert S\right\Vert _{0}^{2}$.
\end{proof}

\begin{definition}
The symmetric bilinear form is given by (\ref{normdef2}) for $\lambda
\in\mathbb{N}_{0}^{N,+},S\in\mathcal{Y}\left(  \tau\right)  $, by $\left\Vert
M_{\alpha,S}\right\Vert ^{2}=\mathcal{E}\left(  \alpha,S\right)
^{-1}\left\Vert M_{\alpha^{+},S}\right\Vert ^{2}$ for $\alpha\in\mathbb{N}%
_{0}^{N}$ and by $\left\langle M_{\alpha,S},M_{\beta,S^{\prime}}\right\rangle
=0$ for $\left(  \alpha,S\right)  \neq\left(  \beta,S^{\prime}\right)  .$
\end{definition}

Next we show that the definition satisfies the hypotheses (BF2).

The step $s_{i}$ with $\alpha_{i}<\alpha_{i+1}$ satisfies (\ref{BF2.2})
because of the value $\dfrac{\mathcal{E}\left(  \alpha.s_{i},S\right)
}{\mathcal{E}\left(  \alpha,S\right)  }$. It remains to check the step with
$\alpha_{i}=\alpha_{i+1}$ and the affine step. The $\left(  i,j\right)
$-product in (\ref{normdef2}) can be written as (note $t^{-1}u\left(
z\right)  =\frac{\left(  1-z/t\right)  \left(  1-tz\right)  }{\left(
1-x\right)  ^{2}}$)%
\[
\prod\limits_{1\leq i<j\leq N}t^{\lambda_{j}-\lambda_{i}}\prod\limits_{l=1}%
^{\lambda_{i}-\lambda_{j}}u\left(  q^{l}t^{c\left(  i,S\right)  -c\left(
j,S\right)  }\right)  .
\]
Suppose $\alpha\in\mathbb{N}_{0}^{N}$ and $\lambda:=\alpha^{+}$; in the
formula for $\mathcal{E}\left(  \alpha,S\right)  $ the condition $\left(
i<j\right)  \&\left(  \alpha_{i}<\alpha_{j}\right)  $ is equivalent to
$\left(  i<j\right)  \&\left(  r_{\alpha}\left(  i\right)  >r_{\alpha}\left(
j\right)  \right)  $. Let $v_{\alpha}=r_{\alpha}^{-1}$ so that $\lambda
_{i}=\alpha_{v_{\alpha}\left(  i\right)  }$ then the product can be indexed by
$\left(  v_{\alpha}\left(  i^{\prime}\right)  <v_{\alpha}\left(  j^{\prime
}\right)  \right)  \&\left(  i^{\prime}>j^{\prime}\right)  $ (where
$i^{\prime}=r_{\alpha}\left(  i\right)  ,j^{\prime}=r_{\alpha}\left(
j\right)  $). Thus
\[
\mathcal{E}\left(  \alpha,S\right)  =\prod\limits_{1\leq j^{\prime}<i^{\prime
}\leq N,v_{\alpha}\left(  i^{\prime}\right)  <v_{\alpha}\left(  j^{\prime
}\right)  }u\left(  q^{\lambda_{i^{\prime}}-\lambda_{j^{\prime}}}t^{c\left(
i^{\prime},S\right)  -c\left(  j^{\prime},S\right)  }\right)  .
\]

\begin{proposition}
Suppose $\alpha_{i}=\alpha_{i+1},j=r_{\alpha}\left(  i\right)  $ and
$m=c\left(  j,S\right)  -c\left(  j+1,S\right)  \geq2$ then $\left\Vert
M_{\alpha,S^{\left(  j\right)  }}\right\Vert ^{2}=\dfrac{\left(
1-t^{1-m}\right)  \left(  t-t^{-m}\right)  }{\left(  1-t^{-m}\right)  ^{2}%
}\left\Vert M_{\alpha,S}\right\Vert ^{2}$.
\end{proposition}

\begin{proof}
By hypothesis $\zeta_{\alpha,S}\left(  i\right)  =q^{\alpha_{i}}t^{c\left(
j,S\right)  }$ and $\zeta_{\alpha,S}\left(  i+1\right)  =q^{\alpha_{i}%
}t^{c\left(  j+1,S\right)  }$ so that $\zeta_{\alpha,S}\left(  i+1\right)
/\zeta_{\alpha,S}\left(  i\right)  =t^{-m}$. Also by Proposition \ref{Sform}
$\left\Vert S^{\left(  j\right)  }\right\Vert _{0}^{2}=u\left(  t^{-m}\right)
\left\Vert S\right\Vert _{0}^{2}$. Suppose first that $\alpha\in\mathbb{N}%
_{0}^{N,+}$ then $j=i$. In the formula for $\left\Vert M_{\alpha,S}\right\Vert
^{2}$ the first product does not change when $S$ is replaced by $S^{\left(
j\right)  }$; the factors $\left(  qt^{c\left(  i,S\right)  };q\right)
_{\lambda_{i}}$,$\left(  qt^{c\left(  i+1,S\right)  };q\right)  _{\lambda_{i}%
}$ trade places. By a similar argument the $\left(  i,j\right)  $-product also
does not change, and $\left\Vert M_{\alpha,S^{\left(  j\right)  }}\right\Vert
^{2}/\left\Vert S^{\left(  j\right)  }\right\Vert _{0}^{2}=\left\Vert
M_{\alpha,S}\right\Vert ^{2}/\left\Vert S\right\Vert _{0}^{2}$. Otherwise
$\alpha\neq\alpha^{+}$ and%
\[
\frac{\left\Vert M_{\alpha,S^{\left(  j\right)  }}\right\Vert ^{2}}{\left\Vert
S^{\left(  j\right)  }\right\Vert _{0}^{2}}=\frac{\left\Vert M_{\alpha
^{+},S^{\left(  j\right)  }}\right\Vert ^{2}}{\mathcal{E}\left(
\alpha,S^{\left(  j\right)  }\right)  \left\Vert S^{\left(  j\right)
}\right\Vert _{0}^{2}}=\frac{\left\Vert M_{\alpha^{+},S}\right\Vert ^{2}%
}{\mathcal{E}\left(  \alpha,S^{\left(  j\right)  }\right)  \left\Vert
S\right\Vert _{0}^{2}}=\frac{\mathcal{E}\left(  \alpha,S\right)  }%
{\mathcal{E}\left(  \alpha,S^{\left(  j\right)  }\right)  }\frac{\left\Vert
M_{\alpha,S}\right\Vert ^{2}}{\left\Vert S\right\Vert _{0}^{2}}.
\]
Recall $\mathcal{E}\left(  \alpha,S\right)  =\prod\limits_{1\leq l<n\leq
N,\alpha_{l}<\alpha_{n}}u\left(  q^{\alpha_{n}-\alpha_{l}}t^{c\left(
r_{\alpha}\left(  n\right)  ,S\right)  -c\left(  r_{\alpha}\left(  l\right)
,S\right)  }\right)  $ and the product does not change when $S$ is replaced by
$S^{\left(  j\right)  }$ (the factors involving $l=i$ or $n=i$ are
interchanged with those involving $l=i+1$ or $n=i+1$). Thus $\mathcal{E}%
\left(  \alpha,S^{\left(  j\right)  }\right)  =\mathcal{E}\left(
\alpha,S\right)  $.
\end{proof}

\begin{proposition}
Suppose $\alpha\in\mathbb{N}_{0}^{N},S\in\mathcal{Y}\left(  \tau\right)  $
then%
\[
\left\Vert M_{\alpha\Phi,S}\right\Vert ^{2}=\dfrac{1-q^{\alpha_{1}%
+1}t^{c\left(  r_{\alpha}\left(  1\right)  ,S\right)  }}{1-q}\left\Vert
M_{\alpha,S}\right\Vert ^{2}.
\]

\end{proposition}

\begin{proof}
We need to compute various ratios of $\mathcal{E}\left(  \alpha,S\right)
,\left\Vert M_{\alpha^{+},S}\right\Vert ^{2},\mathcal{E}\left(  \alpha
\Phi,S\right)  ,\left\Vert M_{\left(  \alpha\Phi\right)  ^{+},S}\right\Vert
^{2}$. Also $r_{\alpha}\left(  i+1\right)  =r_{\alpha\Phi}\left(  i\right)  $
for $1\leq i<N,r_{\alpha}\left(  1\right)  =r_{\alpha\Phi}\left(  N\right)  $.
Let $\lambda:=\alpha^{+}$, then $\lambda_{r_{\alpha}\left(  i\right)  }%
=\alpha_{i}$ for all $i$. Let $m:=r_{\alpha}\left(  1\right)  $. This implies
$\#\left\{  i:\alpha_{i}>\alpha_{1}\right\}  =m-1$, thus $\lambda
_{m-1}>\lambda_{m}$ and $\left(  \alpha\Phi\right)  _{m}^{+}=\lambda_{m}+1$.
Also $k\left(  \left(  \alpha\Phi\right)  _{m}^{+}\right)  =k\left(
\lambda\right)  =N-2m+1$. This implies%
\begin{align*}
\frac{\left\Vert M_{\left(  \alpha\Phi\right)  ^{+},S}\right\Vert ^{2}%
}{\left\Vert M_{\lambda,S}\right\Vert ^{2}}  &  =t^{N-2m+1}\frac
{1-q^{\lambda_{m}+1}t^{c\left(  m,S\right)  }}{1-q}t^{m-1}\prod\limits_{i=1}%
^{m-1}u\left(  q^{\lambda_{i}-\lambda_{m}}t^{c\left(  i,S\right)  -c\left(
m,S\right)  }\right)  ^{-1}\\
&  \times t^{m-N}\prod\limits_{j=m+1}^{N}u\left(  q^{\lambda_{m}+1-\lambda
_{i}}t^{c\left(  m,S\right)  -c\left(  j,S\right)  }\right)  .
\end{align*}
Let $\mu:=\left(  \alpha\Phi\right)  ^{+}$, then $\mathcal{E}\left(
\alpha\Phi,S\right)  =\prod\limits_{i<j,v_{\Phi}\left(  i\right)
>v_{\alpha\Phi}\left(  j\right)  }u\left(  q^{\mu_{i}-\mu_{j}}t^{c\left(
i,S\right)  -c\left(  j,S\right)  }\right)  $ and $\mathcal{E}\left(
\alpha,S\right)  =\prod\limits_{i<j,v_{\alpha}\left(  i\right)  >v_{\alpha
}\left(  j\right)  }u\left(  q^{\lambda_{i}-\lambda_{j}}t^{c\left(
i,S\right)  -c\left(  j,S\right)  }\right)  .$ From $v_{\alpha}\left(
i\right)  =v_{\alpha\Phi}\left(  i\right)  +1$ except $v_{\alpha}\left(
m\right)  =1,v_{\alpha\Phi}\left(  m\right)  =N$ the inversions $\left(
i<j\right)  \&\left(  v_{\alpha}\left(  i\right)  >v_{\alpha}\left(  j\right)
\right)  $ occur in both the products provided that $j\neq m$ in which case
the pairs $\left\{  \left(  i,m\right)  :1\leq i\leq m-1\right\}  $ do not
occur in $\mathcal{E}\left(  \alpha\Phi,S\right)  $, or if $i=m$ and the pairs
$\left\{  \left(  m,j\right)  :m<j\leq N\right\}  $ do not occur in
$\mathcal{E}\left(  \alpha,S\right)  $. Also $\mu_{i}=\lambda_{i}$ for all $i$
except $\mu_{m}=\lambda_{m}+1$. Thus%
\[
\frac{\mathcal{E}\left(  \alpha,S\right)  }{\mathcal{E}\left(  \alpha
\Phi,S\right)  }=\prod\limits_{i=1}^{m-1}u\left(  q^{\lambda_{i}-\lambda_{m}%
}t^{c\left(  i,S\right)  -c\left(  m,S\right)  }\right)  \prod\limits_{j=m+1}%
^{N}u\left(  q^{\lambda_{m}+1-\lambda_{j}}t^{c\left(  m,S\right)  -c\left(
j,S\right)  }\right)  ^{-1},
\]
and%
\[
\frac{\left\Vert M_{\alpha\Phi,S}\right\Vert ^{2}}{\left\Vert M_{\alpha
,S}\right\Vert ^{2}}=\frac{\left\Vert M_{\left(  \alpha\Phi\right)  ^{+}%
,S}\right\Vert ^{2}}{\left\Vert M_{\lambda,S}\right\Vert ^{2}}\frac
{\mathcal{E}\left(  \alpha,S\right)  }{\mathcal{E}\left(  \alpha\Phi,S\right)
}=\frac{1-q^{\lambda_{m}+1}t^{c\left(  m,S\right)  }}{1-q}.
\]
Finally $\zeta_{\alpha,S}\left(  1\right)  =q^{\alpha_{1}}t^{c\left(
r_{\alpha}\left(  1\right)  ,S\right)  }=q^{\lambda_{m}}t^{c\left(
m,S\right)  }$.
\end{proof}

\begin{corollary}
The bilinear form satisfies (\ref{BF2.4}).
\end{corollary}

\begin{proof}
By Lemma \ref{Mdiff2}
\begin{align*}
\left\langle M_{\alpha\Phi,S}\mathcal{D}_{N},M_{\alpha,S}\left(
\boldsymbol{w}^{\ast}\right)  ^{-1}\right\rangle  &  =\left(  1-q\zeta
_{\alpha,S}\left(  1\right)  \right)  \left\langle M_{\alpha,S}\boldsymbol{w,}%
M_{\alpha,S}\left(  \boldsymbol{w}^{\ast}\right)  ^{-1}\right\rangle \\
&  =\left(  1-q\zeta_{\alpha,S}\left(  1\right)  \right)  \left\Vert
M_{\alpha,S}\right\Vert ^{2},\\
\left\langle M_{\alpha\Phi,S},x_{N}\left(  M_{\alpha,S}\left(  \boldsymbol{w}%
^{\ast}\right)  ^{-1}\right)  \boldsymbol{w}^{\ast}\boldsymbol{w}%
\right\rangle  &  =\left\langle M_{\alpha\Phi,S},M_{\alpha\Phi,S}\right\rangle
=\dfrac{1-q\zeta_{\alpha,S}\left(  1\right)  }{1-q}\left\Vert M_{\alpha
,S}\right\Vert ^{2}%
\end{align*}
by the Proposition, thus $\left(  1-q\right)  \left\langle M_{\alpha\Phi
,S},x_{N}g\boldsymbol{w}^{\ast}\boldsymbol{w}\right\rangle =\left\langle
M_{\alpha\Phi,S}\mathcal{D}_{N},g\right\rangle $ when $g=M_{\alpha,S}\left(
\boldsymbol{w}^{\ast}\right)  ^{-1}$. It suffices to prove $\left\langle
f\mathcal{D}_{N},g\right\rangle =\left(  1-q\right)  \left\langle
f,x_{N}\left(  g\boldsymbol{w}^{\ast}\boldsymbol{w}\right)  \right\rangle $
for $f=M_{\gamma,S}$ and $g\boldsymbol{w}^{\ast}=M_{\beta,S^{\prime}}$ with
$\left\vert \gamma\right\vert =\left\vert \beta\right\vert +1$. If $\gamma
_{N}=0$ then $M_{\gamma,S}\mathcal{D}_{N}=0$ and $\left\langle M_{\gamma
,S}\mathcal{D}_{N},M_{\beta.S^{\prime}}\left(  \boldsymbol{w}^{\ast}\right)
^{-1}\right\rangle =0$ while $\left\langle M_{\gamma,S},x_{N}\left(
M_{\beta,S^{\prime}}\boldsymbol{w}\right)  \right\rangle =\left\langle
M_{\gamma,S},M_{\beta\Phi,S^{\prime}}\right\rangle =0$ because $\gamma
\neq\beta\Phi$. If $\gamma=\alpha\Phi$ for some $\alpha$ with $\left(
\alpha,S\right)  \neq\left(  \beta,S^{\prime}\right)  $ then $\left\langle
M_{\alpha\Phi,S},x_{N}\left(  M_{\beta,S^{\prime}}\boldsymbol{w}\right)
\right\rangle =\left\langle M_{\alpha\Phi,S},M_{\beta\Phi,S^{\prime}%
}\right\rangle =0$ and $\left\langle M_{\alpha\Phi,S}\mathcal{D}_{N}%
,M_{\beta.S^{\prime}}\left(  \boldsymbol{w}^{\ast}\right)  ^{-1}\right\rangle
=\allowbreak\left(  1-q\zeta_{\alpha}\left(  1\right)  \right)  \left\langle
M_{\alpha,S}\boldsymbol{w,}M_{\beta.S^{\prime}}\left(  \boldsymbol{w}^{\ast
}\right)  ^{-1}\right\rangle =\left(  1-q\zeta_{\alpha}\left(  1\right)
\right)  \left\langle M_{\alpha,S}\boldsymbol{,}M_{\beta.S^{\prime}%
}\right\rangle =0$. The case $\left(  \alpha,S\right)  =\left(  \beta
,S^{\prime}\right)  $ is already done.
\end{proof}

\begin{proposition}
\label{regpos2}Suppose $\dim V_{\tau}\geq2$, $q,t>0$ and $q\neq1$ then the
form $\left\langle \cdot,\cdot\right\rangle $ is positive-definite provided
$0<q<\min\left(  t^{-h_{\tau}},t^{h_{\tau}}\right)  $ or $q>\max\left(
t^{-h_{\tau}},t^{h_{\tau}}\right)  $; equivalently $\min\left(  q^{-1/h_{\tau
}},q^{1/h_{\tau}}\right)  <t<\max\left(  q^{-1/h_{\tau}},q^{1/h_{\tau}%
}\right)  $.
\end{proposition}

\begin{proof}
In the definition of $\left\langle M_{\alpha,S},M_{\alpha,S}\right\rangle $
there is an even number of factors of the form $1-q^{a}t^{b}$ where
$a=1,2,3,\ldots$ and $b$ is one of $c\left(  i,S\right)  $, $c\left(
i,S\right)  -c\left(  j,S\right)  $, or $c\left(  i,S\right)  -c\left(
j,S\right)  \pm1$. The $c\left(  i,S\right)  $ values lie in $\left[
1-\ell\left(  \tau\right)  ,\tau_{1}-1\right]  $; thus $-h_{\tau}\leq b\leq
h_{\tau}$ where $h_{\tau}=\tau_{1}+\ell\left(  \tau\right)  -1$, the maximum
hook length in the Ferrers diagram $\lambda$. Consider the four cases (1)
$0<q<1,0<t<1$, then $q^{a}t^{b}\leq qt^{-h_{\tau}}<1$ provided $q<t^{h_{\tau}%
}$; (2) $0<q<1,t\geq1$, then $q^{a}t^{b}\leq qt^{h_{\tau}}<1$ provided
$q<t^{-h_{\tau}}$; (3) $q>1,0<t<1,$ then $q^{a}t^{b}\geq qt^{h_{\tau}}>1$
provided $q>t^{-h_{\tau}}$; (4) $q>1,t\geq1$, then $q^{q}t^{b}\geq
qt^{-h_{\gamma}}>1$ provided $q>t^{h_{\tau}}$. Thus $\left\Vert M_{\alpha
,S}\right\Vert ^{2}>0$ if $\min\left(  q^{-1/h_{\tau}},q^{1/h_{\tau}}\right)
<t<\max\left(  q^{-1/h_{\tau}},q^{1/h_{\tau}}\right)  $.
\end{proof}

There is an illustration in Figure \ref{posit2} with $h_{\tau}=3$ (for
$\tau=\left(  2,1\right)  $ or $\tau=\left(  2,2\right)  $).

From a similar argument it follows that the transformation formulas for
Macdonald polynomials have no poles when $\min\left(  q^{-1/k},q^{1/k}\right)
<t<\max\left(  q^{-1/k},q^{1/k}\right)  $ with $k=h_{\tau}-1$.%

\begin{figure}
[ptb]
\begin{center}
\includegraphics[
height=3.1514in,
width=3.1514in
]%
{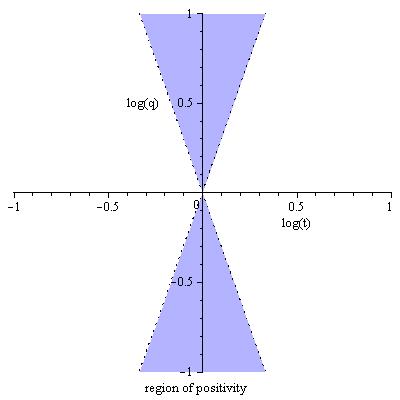}%
\caption{Logarithmic coordinates, h = 3}%
\label{posit2}%
\end{center}
\end{figure}

\subsection{Singular polynomials}

A singular polynomial $f\in\mathcal{P}_{\tau}$ is one which satisfies
$f\mathcal{D}_{i}=0$ for all $i$ when $\left(  q,t\right)  $ are specialized
to some specific relation of the form $q^{a}t^{b}=1$. By Proposition
\ref{fDxg} $f$ satisfies $\left\langle f,g\right\rangle =0$ for all
$g\in\mathcal{P}_{\tau}$, and in particular $\left\langle f,f\right\rangle
=0$. Thus the singular polynomial phenomenon can not occur in the $\left(
q,t\right)  $-region of positivity. The boundary of the region does allow
singular polynomials. There are Macdonald polynomials which are singular when
specialized to $q=t^{h_{\tau}}$ or $q=t^{-h_{\tau}}$. These values do not
produce poles in the polynomial coefficients as remarked above, since
$\frac{1}{h_{\tau}}<\frac{1}{h_{\tau}-1}$.

\begin{proposition}
\label{zerodiff}Suppose $\alpha\in\mathbb{N}_{0}^{N},S\in\mathcal{Y}\left(
\tau\right)  $ and $\alpha_{i}=0$ for $m<i\leq N$ then $M_{\alpha
,S}\mathcal{D}_{j}=0$ for $m<j\leq N$.
\end{proposition}

\begin{proof}
Arguing by induction the start is $M_{\alpha,S}\mathcal{D}_{N}=\left(
1/x_{N}\right)  M_{\alpha,S}\left(  1-\xi_{N}\right)  =\left(  1/x_{N}\right)
\left(  1-\zeta_{\alpha,S}\left(  N\right)  \right)  M_{\alpha,S}=0$, since
$\zeta_{\alpha,S}\left(  N\right)  =1$. Suppose now that $\beta_{i}=0$ for
$i\geq k+1$ implies $M_{\beta,S^{\prime}}\mathcal{D}_{j}=0$ for $j\geq k+1$
and any $\left(  \beta,S^{\prime}\right)  $. Suppose $\alpha_{i}=0$ for $i\geq
k$ then $r_{\alpha}\left(  i\right)  =i$ and $\zeta_{\alpha,S}\left(
i\right)  =t^{r\left(  i,S\right)  }$ for $i\geq k$ and $M_{\alpha
,S}\boldsymbol{T}_{k}$ is one of $tM_{\alpha,S}$, $-M_{\alpha,S}$,
$M_{\alpha,S^{\left(  k\right)  }}-\frac{t-1}{\rho-1}M_{\alpha,S}$,
$\frac{\left(  1-t\rho\right)  \left(  t-\rho\right)  }{\left(  1-\rho\right)
^{2}}M_{\alpha,S^{\left(  k\right)  }}-\frac{t-1}{\left(  1-\rho\right)
}M_{\alpha,S}$ depending on $c\left(  k+1,S\right)  -c\left(  k,S\right)  =1$,
$=-1$, $\geq2$, $\leq-2$ respectively and $\rho=t^{c\left(  k+1,S\right)
-c\left(  k,S\right)  }$. Then $M_{\alpha,S}\mathcal{D}_{k}=\frac{1}{t}\left(
M_{\alpha,S}\boldsymbol{T}_{k}\right)  \mathcal{D}_{k+1}\boldsymbol{T}_{k}$
and $\left(  M_{\alpha,S}\boldsymbol{T}_{k}\right)  \mathcal{D}_{k+1}=0$ by
the inductive hypothesis.
\end{proof}

\begin{lemma}
Suppose $\alpha=\left(  \alpha_{1},\ldots,\alpha_{m-1},1,0,\ldots\right)  $
with $\alpha_{i}\geq1$ for $i\leq m$ and $S\in\mathcal{Y}\left(  \tau\right)
$ then
\[
M_{\alpha,S}\mathcal{D}_{m}=t^{m-N}\prod\limits_{j=m}^{N-1}u\left(
qt^{c\left(  m,S\right)  -c\left(  j+1,S\right)  }\right)  M_{\beta^{\left(
N\right)  },S}\mathcal{D}_{N}\boldsymbol{T}_{N-1}\cdots\boldsymbol{T}_{m},
\]
where $\alpha^{\left(  N\right)  }=\left(  \alpha_{1},\ldots,\alpha
_{m-1},0,0,\ldots,1\right)  $.
\end{lemma}

\begin{proof}
For $m\leq j\leq N$ let $\alpha^{\left(  j\right)  }=\left(  \alpha_{1}%
,\ldots,\alpha_{m-1},0,\ldots,\overset{j}{1},0\ldots\right)  $ so that
$\alpha_{i}^{\left(  j\right)  }=\alpha_{i}$ except $\alpha_{j}^{\left(
j\right)  }=1$ and $\alpha_{m}^{\left(  j\right)  }=0$ (when $j\neq m$). Then
$\zeta_{\alpha^{\left(  j\right)  },S}\left(  j\right)  =qt^{c\left(
m,S\right)  }$ and $\zeta_{\alpha^{\left(  j\right)  },S}\left(  j+1\right)
=t^{c\left(  j+1,S\right)  }$ (since $r_{\alpha^{\left(  m\right)  }}\left(
j\right)  =m$) and
\[
M_{\alpha^{\left(  j\right)  },S}\boldsymbol{T}_{j}=\frac{\left(
1-t\rho\right)  \left(  t-\rho\right)  }{\left(  1-\rho\right)  ^{2}}%
M_{\alpha^{\left(  j+1\right)  },S}+\frac{\rho\left(  1-t\right)  }{\left(
1-\rho\right)  }M_{\alpha^{\left(  j\right)  },S}%
\]
from (\ref{MT<i}) with $\rho=qt^{c\left(  m,S\right)  -c\left(  j+1,S\right)
}$. Thus
\[
M_{\alpha^{\left(  j\right)  },S}\mathcal{D}_{j}=\frac{1}{t}M_{\alpha^{\left(
j\right)  },S}\boldsymbol{T}_{j}\mathcal{D}_{j+1}\boldsymbol{T}_{j}=\frac
{1}{t}u\left(  qt^{c\left(  m,S\right)  -c\left(  j+1,S\right)  }\right)
M_{\alpha^{\left(  j+1\right)  },S}\mathcal{D}_{j+1}\boldsymbol{T}_{j},
\]
because $M_{\alpha^{\left(  j\right)  },S}\mathcal{D}_{j+1}=0$. Iterate this
formula starting with $j=m$ and $\alpha^{\left(  m\right)  }=\alpha$, ending
with $j=N-1$ to obtain the stated formula.
\end{proof}

Recall that $S_{1}$ is the $\mathrm{inv}$-minimal RSYT with the numbers
$N,N-1,N-2,\ldots,1$ entered row-by-row and let $l=\ell\left(  \tau\right)  $,
$\alpha=\left(  1^{\tau_{l}},0^{N-\tau_{l}}\right)  $ thus the entry at
$\left(  l,1\right)  $ is $\tau_{l}$ and $c\left(  \tau_{l},S_{1}\right)
=1-l$. The entry at $\left(  1,\tau_{1}\right)  $ is $N-\tau_{1}+1$ and
$c\left(  N-\tau_{1}+1,S_{1}\right)  =\tau_{1}-1$.

\begin{proposition}
$M_{\alpha,S_{1}}$ is singular for $q=t^{h_{\tau}}$.
\end{proposition}

\begin{proof}
By the lemma with $m=\tau_{l}$,
\[
M_{\alpha,S_{1}}\mathcal{D}_{\tau_{l}}=t^{\tau_{l}-N}\prod\limits_{j=\tau_{l}%
}^{N-1}u\left(  qt^{1-l-c\left(  j+1,S_{1}\right)  }\right)  M_{\alpha
^{\left(  N\right)  },S_{1}}\mathcal{D}_{N}\boldsymbol{T}_{N-1}\cdots
\boldsymbol{T}_{\tau_{l}}.
\]
The factors in the denominator of the product are of the form
$1-qt^{1-l-c\left(  j+1,S_{1}\right)  }$ with $c\left(  j+1,S_{1}\right)
\leq\tau_{1}-1$ so that $1-l-c\left(  j+1,S_{1}\right)  \geq2-l-\tau
_{1}=1-h_{\tau}>h_{\tau}$. Furthermore the numerator factor at $j=N-\tau_{1}$
is $\left(  t-qt^{2-l-\tau_{1}}\right)  \left(  1-qt^{3-l-\tau_{1}}\right)  $
which vanishes at $qt^{-h_{\tau}}=1$. By Proposition \ref{zerodiff}
$M_{\alpha,S_{1}}\mathcal{D}_{i}=0$ for $i>\tau_{l}$. If $1\leq i<\tau_{l}$
then $M_{\alpha,S_{1}}\boldsymbol{T}_{i}=tM_{\alpha,S_{1}}$ (because $i,i+1$
are in the same row of $S_{1}).$thus%
\begin{align*}
M_{\alpha,S_{1}}\mathcal{D}_{i}  &  =t^{i-\tau_{l}}M_{\alpha,S_{i}%
}\boldsymbol{T}_{i}\boldsymbol{T}_{i+1}\cdots\boldsymbol{T}_{\tau_{l}%
-1}\mathcal{D}_{\tau_{l}}\boldsymbol{T}_{\tau_{l}-1}\cdots\boldsymbol{T}_{i}\\
&  =M_{\alpha,S_{i}}\mathcal{D}_{\tau_{l}}\boldsymbol{T}_{\tau_{l}-1}%
\cdots\boldsymbol{T}_{i}=0
\end{align*}
when $q=t^{h_{\tau}}$.
\end{proof}

We apply the same argument to $S_{0}$ where the numbers $N,N-1,\ldots,1$ are
entered column-by-column. Let $m=\tau_{\tau_{1}}^{\prime}$, that is, the
length of the last column of $\tau$. Then the entry at $\left(  \tau
_{1},1\right)  $ is $m$ and $c\left(  m,S_{0}\right)  =\tau_{1}-1$. Also the
entry at $\left(  l,1\right)  $ is $N-l+1$ and $c\left(  N-l+1\right)  =1-l$.

\begin{proposition}
Set $\alpha=\left(  1^{m},0^{N-m}\right)  $ then $M_{\alpha,S_{0}}$ is
singular for $q=t^{-h_{\tau}}$.
\end{proposition}

\begin{proof}
By the lemma%
\[
M_{\alpha,S_{0}}\mathcal{D}_{m}=t^{m-N}\prod\limits_{j=m}^{N-1}u\left(
qt^{\tau_{1}-1-c\left(  j+1,S_{1}\right)  }\right)  M_{\alpha^{\left(
N\right)  },S_{0}}\mathcal{D}_{N}\boldsymbol{T}_{N-1}\cdots\boldsymbol{T}%
_{m}.
\]
The factors in the denominator of the product are of the form $1-qt^{\tau
_{1}-1-c\left(  j+1,S_{1}\right)  }$ with $c\left(  j+1,S_{1}\right)  \geq1-l$
so that $\tau_{1}-1-c\left(  j+1,S_{1}\right)  \leq\tau_{1}+l-2<h_{\tau}$.
Furthermore the numerator factor at $j=N-l$ is $\left(  t-qt^{\tau_{1}%
+l-2}\right)  \left(  1-qt^{\tau_{1}+l-1}\right)  $ which vanishes at
$qt^{h_{\tau}}=1$. The rest of the argument is as in the previous proposition
with the difference that $M_{\alpha,S_{0}}\boldsymbol{T}_{i}=-M_{\alpha,S_{0}%
}$ for $1\leq i<m$.
\end{proof}

In conclusion we have constructed a symmetric bilinear form on $\mathcal{P}%
_{\tau}$ for which the operators $T_{i}$ and $\boldsymbol{\xi}_{i}$ are
self-adjoint, the Macdonald polynomials $M_{\alpha,S}$ are mutually
orthogonal, and the form is positive-definite for $q>0,q\neq1$ and
$\min\left(  q^{-1/h_{\tau}},q^{1/h_{\tau}}\right)  <t<\max\left(
q^{-1/h_{\tau}},q^{1/h_{\tau}}\right)  $) where $h_{\tau}=\tau_{1}+\ell\left(
\tau\right)  -1$. The bound is sharp, as demonstrated by the existence of
singular polynomials for $q=t^{\pm h_{\tau}}$.

\end{document}